\documentclass[12pt,a4paper]{article}

\usepackage[T1]{fontenc}
\usepackage[utf8]{inputenc}

\usepackage{textcomp}
\usepackage[official]{eurosym}
\usepackage{amssymb,mathrsfs}
\usepackage{amsthm}
\usepackage{mathtools}
\usepackage{latexsym}
\usepackage{amsmath}
\usepackage{tikz}

\usepackage{ifthen}

\newtheorem{theorem}{Theorem}[section]
\newtheorem{lemma}[theorem]{Lemma}
\newtheorem{proposition}[theorem]{Proposition}
\newtheorem{corollary}[theorem]{Corollary}
\newtheorem{remark}[theorem]{Remark}
\newtheorem{question}{Question}
\newtheorem{definition}[theorem]{Definition}
\newtheorem{example}{Example}
\newcommand{\2}{NONEMPTY}
\newcommand{\1}{EMPTY}
\newcommand{\kare}{$\square$}

\begin{document}
	\title{Noetherian $\pi$-bases and  Telg{\'a}rsky's Conjecture}
	\author{Servet Soyarslan$^{(a)}$\footnote{Corresponding author. \newline E-mail addresses: servet.soyarslan@gmail.com (S. Soyarslan),  osul@metu.edu.tr(S. \"Onal).} , S\"uleyman \"Onal$^{(b)}$ }
	
	\maketitle	
	
	{\scriptsize a. Independent researcher, Ankara, Türkiye.
		
		 b. Middle East Technical University, Department of Mathematics, 06531 Ankara, Türkiye.}

\begin{abstract}
 We investigate  Noetherian families and show that every topological space has a Noetherian $\pi$-base. We prove that if a topological space has  some special Noetherian $\pi$-bases, then \2 has a 2-tactic in the Banach-Mazur game on a space $X$, denoted as $BM(X)$, whenever \2 has a winning strategy in $BM(X)$. 
 This result encompasses an important theorem of Galvin in this context and is related to Telg{\'a}rsky's  conjecture on this subject. One of our examples is that  any space $X$  with $\pi w(X)\leq \omega_1$ has this special Noetherian $\pi$-base.  We pose some questions about this topic.
\end{abstract}

	

\section{Introduction}

This article  consists of some  results about Noetherian families and \2's strategies in the Banach Mazur game, along with establishing connections between them. 

Through the examination of Noetherian families in topological spaces,   we prove that every topological space has a Noetherian $\pi$-base. Because this is a highly general result, it's possible that someone has previously proven it. Nevertheless, we did not encounter this result elsewhere.

We represent the Noetherian $\pi$-base of a topological space with a matrix-like table, which we call the Noetherian table. By using a Noetherian $\pi$-base which has a "special" Noetherian table, we investigate \2's 2-tactic in $BM(X)$.

Historically,  
Telg{\'a}rsky mentioned that it is an unsettled question (posed by Debs) whether a winning
strategy of \2 in $BM(X)$ can be reduced to a   2-tactic.  He  also conjectured that
a $(k+1)$-tactic of \2  cannot be reduced, in general, to a  $k$-tactic (\cite{telgarsky1987topological}, p. 236).

 It is known that if \2 has a winning strategy in  $BM(X)$ where $X$ has a $\pi$-base $\mathcal{B}$ with the following (*)-property, then \2 has a   2-tactic (this is Galvin's theorem (unpublished), see Theorem 37 in \cite{bartoszynski1993covering}, for the  proof see  Theorem 2.4 in \cite{brian2021telgarsky}).  

(*)-property: 
For any $B\in \mathcal{B}$  there exists a cellular family 
$\mathcal{A}_B\subseteq \tau(B)$  such that $|\mathcal{A}_B|\geq  |\{V \in\mathcal{B}: B \subseteq V \}|$.

(If any $\pi$-base $\mathcal{B}$  satisfies  the (*)-property, then we call $\mathcal{B}$   \textit{a Galvin $\pi$-base}.)

To provide a proof of Galvin's theorem,  in Theorem 2.4  of \cite{brian2021telgarsky}  the authors utilize  the following (*')-property  instead of the (*)-property.   

(*')-property:  For any $O\in \tau^*(X)$, there exists a $\pi$-base \(\mathcal{B}_O\) for the subspace \( O \) such that for all \( W \in \tau^*(O) \), there exists a  disjoint family $\mathcal{A}\subseteq \tau^*(O)$ satisfying $|\{ B \in \mathcal{B}_O : W \subseteq B \}|\leq |\mathcal{A}|.$

In the comment after that Theorem
2.4 on page 6 in \cite{brian2021telgarsky}, the authors said that  (*')-property can be weakened slightly and the weak version (the weak version is (ii) of our Theorem  \ref{ne}) can be utilized in their Theorem  2.4  instead of (*')-property.  In Theorem \ref{ne} of this paper, we establish the equivalence of the properties (*),  (*') and (ii) for any space X. Consequently, Theorem 2.4 and the subsequent comment in \cite{brian2021telgarsky} are shown to be equivalent to Galvin's theorem. 

Our main theorem (see Theorem \ref{Teorem winning ten 2-tactige ana sonuc 2})  states that if \2 has a winning strategy  in the Banach-Mazur game on a space $X$ which has a  Noetherian $\pi$-base with "special" Noetherian table, then \2 has a 2-tactic in this game. We show that  Galvin's this theorem can be given as a corollary of our main theorem (see Corollary \ref{vin}). 

We think that the scope of our main theorem is possible broader than Galvin's theorem (and so Theorem 2.4 in \cite{brian2021telgarsky}). 
Because   for any $B\in \mathcal{B}$ Galvin's theorem hypothesizes that $B$ has to have as many disjoint open subsets as its open supersets, while our main theorem $B$ has to have as many disjoint open subsets as some (not all)  of  its open supersets.
 Because our theorem uses a Noetherian table. 
  $B$ is in this Noetherian table and  it is sufficient that if    $B$ has as many disjoint open subsets as  some of open supersets which are in the upper-left side of   $B$ in the Noetherian table. Note that for Galvin's theorem,  $B$ has as many disjoint open subsets as  open supersets which are in both side (upper-left and upper-right sides) of   $B$ in the Noetherian table. We have provided a base for the real numbers that reveals this difference. Namely, the real numbers has a $\pi$-base $\mathcal{B}$ such that $\mathcal{B}$ is not a Galvin base  and $\mathcal{B}$ has the conditions of our main theorem (see Example \ref{bitsin}).     However, the question  "Is there any space $X$ having a $\pi$-base satisfying conditions of our main theorem, but $X$  has no Galvin $\pi$-base?" remains open 
 (see  Question \ref{Galvin}).

A space $X$ is called quasi-regular  if $X$ is Hausdorff and, for
every nonempty open $U\subseteq X$, there is a nonempty open $V\subseteq X$ satisfying $\overline{V}\subseteq U$. Clearly, every $T_3$ space is quasi-regular. In \cite{brian2021telgarsky}, the authors show that  existence of Galvin $\pi$-bases in  quasi-regular spaces equivalent the set theoretic statement $\nabla$ (see Theorem 2.10 and 2.4 in \cite{brian2021telgarsky}). So, they proved that under some set theoretic assumption every $T_3$-space $X$ has  a Galvin $\pi$-base and so, \2 has a  2-tactic if \2 has a winning strategy in $BM(X)$. But this does not mean that we can say these things about \2 only when the Galvin $\pi$-basis exists. We asked a question about this (see Question \ref{de}). In other words, such research on the Telg{\'a}rsky's Conjecture  (even for quasi-regular spaces) is not yet finished. They noted that the “quasi-regular” hypothesis in Theorem 2.10 (in \cite{brian2021telgarsky}) cannot be omitted. They showed that if $X$ is quasi-regular space with $\pi w(X)\leq \omega_1$ and \2 has a winning strategy  in $BM(X)$, then  \2 has a 2-tactic  in $BM(X)$ (see Remark 2.9 in \cite{brian2021telgarsky}). We proved this result for all topological spaces not just quasi-regular spaces (see Theorem \ref{C:2}).

   Assume that no regular limit cardinal exists, and let $X$ denote either a linearly ordered space or, more generally, a generalized ordered space; we show that if \2 has a winning strategy in $BM(X)$, then \2 also has a 2-tactic in $BM(X)$.  
 
\vspace{2mm}\textbf{Notations and Terminology}\vspace{2mm}

In this paper, for a topological space, unless otherwise mentioned, we do not assume any separation axioms.
The family of nonempty elements of the topology on a space $X$ is denoted by $\tau^*(X)$.
 $pur(\mathcal{A})$ is the set $\{U\in \mathcal{A}: U$ has no isolated points and any nonempty open subset $W$ of $U$ has a cellular family of cardinality $\omega\}$ for any $\mathcal{A}\subseteq \tau^*(X)$.
 Any base or $\pi$-base, in this article for a topological space $X$, does not contain the  empty-set.  For definition of \textit{Galvin $\pi$-base}, see the  introduction. $\mathsf{CH}$ and $\mathsf{GCH}$ stand for continuum hypothesis and generalized continuum hypothesis, respectively. The first infinite ordinal is denoted by $\omega$ and the first uncountable cardinal is denoted by $\omega_1$.  Where $X$ is a topological space, we let $w(X)$,  $\pi w(X)$, $d(X)$ denote weight, $\pi$-weight and density of $X$, respectively. A cellular family of a topological space $X$ is a  disjoint subfamily of $\tau^*(X)$. The cellularity  of a space $X$, denoted $c(X)$, is defined as  $c(X)=\sup\{|\mathcal{A}|:\mathcal{A}$ is a cellular family in $X\} $.  
    The Souslin  number   of a space $X$, denoted $S(X)$, is defined as  $S(X)=\min\{\kappa: X$ has no cellular family of size $ \kappa \}$. 
   When $A$ is a subset of a space X, we let $\overline{A}$ denote the closure  of $A$ in $X$.      
   Where $A$ is a set, we let $\mathscr{P}(A)$, $|A|$ denote the power set and cardinality of $A$, respectively. If we write a set $A$ as $A=\{a_i:i\in \alpha\}$ where $\alpha$ is a  cardinal  or ordinal number, then we always suppose that  this matching is one to one, i.e., $a_i\neq a_j$ if $i\neq j$. The rest of  terminology used in    this paper is standard and can be found in  \cite{engelking1989general}.

\vspace{2mm}\textbf{The banach Mazur game and the  strategies for \2}\vspace{2mm}

We think a topological space $X$ and two players.  The two players  are called different names by some authors, in this paper,  we prefer to call them \1 and \2. In the round 1, \1, who always goes first, chooses a  $U_0\in \tau^*(X)$. Then, \2 chooses a nonempty   open subset  $V_0\subseteq U_0$. Thus, the  round 1 is completed. In  any further round, where $\omega> n>0$, \1  chooses a nonempty   open subset $U_n\subseteq V_{n-1}$ and \2 responds with a nonempty   open subset  $V_n\subseteq U_n$. Thus,  the sequence, which is called \textit{a play}, $(U_0, V_0, U_1, V_1, \ldots,U_n, V_n,\ldots)$ is formed.  If $\bigcap_{i\in\omega} U_i=\emptyset$,  then \1 wins this play. If not, \2 wins.  The family consisting of all these kinds of plays is called the \textit{Banach-Mazur game}  and denoted by $BM(X)$.

A finite sequence $(U_0,V_0,U_1,V_1,\ldots, U_n, V_n)$  consisting of legal moves in  $BM(X)$ is called   $n$-\textit{play part} or just \textit{a play part} in $BM(X)$. A finite sequence $(U_0,V_0,U_1,V_1,\ldots, U_{n-1},V_{n-1},U_n)$  consisting of legal moves in  $BM(X)$ is called   $n$-\textit{added play part} or just \textit{an added play part} in $BM(X)$. 

\textit{A winning strategy} for \2 in $BM(X)$ can be thought of as  a function $\sigma$ which  determines how to make a move  in  each round of a play and produces a win for \2. More clearly, the domain of $\sigma$ is this set:  $\{(U_0, V_0, U_1, V_1, \ldots,U_n):(U_0, V_0, U_1, V_1, \ldots,U_n)$ is an added play part in $BM(X)\}$ and range of $\sigma$ is the set which consists of all possible moves of \2.   If in any play \2 uses this wining strategy $\sigma$ this play will be as follows: $$(U_0,V_0,U_1,V_1,U_2,V_2,U_3,V_3,\ldots)$$  in this play  $\sigma((U_0,V_0,\ldots,U_i))=V_i$ for all $i\in \omega$ and
 $\bigcap_{i\in \omega}U_i\neq \emptyset$. 

NOTE THAT many authors use the term "winning k-tactic" TO DISTINGUISH strategies that guarantee a win from those that do not. HOWEVER, FOR BREVITY, we will SIMPLY USE k-tactic throughout this paper, as we EXCLUSIVELY CONSIDER strategies that lead to victory for Player 2.

A  $k$-\textit{tactic}  for \2 similarly is a  function   that produces a win for \2 and it is defined by   using the  previous $k$ moves of \1. A  1-tactic is sometimes called \textit{a stationary strategy}  (Most of the  authors use the term "winning $k$-\textit{tactic}"  rather than  "$k$-\textit{tactic}". Because they use the term "$k$-\textit{tactic}" for the strategy that doesn't require a win for a player. In this article, for brevity, we use the term  "$k$-\textit{tactic}" instead of "winning $k$-\textit{tactic}" because we do not need the strategy that doesn't require a win).

For more information about the  Banach-Mazur game and some other topological games, see   \cite{telgarsky1987topological} or \cite{galvin1986stationary},  for recent results see \cite{onal2020strategies}.

\section{Some  properties of   Noetherian families}

The concept of Noetherian family is known. The rank of Noetherian families, which is similar to rank of well founded relations in set theory,  was defined  in \cite{onal2020strategies}. These concepts   were used   in \cite{onal2020strategies} to get some results about topological games. In this paper, we use some properties of Noetherian families and,
 especially, Noetherian $\pi$-bases. Seemingly, Noetherian $\pi$-bases and their ranks are useful for topological games so we think it might be good to do some study on them. In this section, we define the \textit{Noetherian table} concept and investigate some properties of Noetherian $\pi$-bases and their ranks.

\begin{definition}\label{Noetherian} A family  $\mathcal{A}\neq \emptyset$ is called \textbf{Noetherian} if $\emptyset\notin \mathcal{A}$ and every strictly increasing sequence  $A_1\subsetneqq A_2 \subsetneqq A_3\subsetneqq \ldots$ of elements of $\mathcal{A}$ is finite.
	
	It is easy to check that $\mathcal{A}$ is Noetherian iff for all nonempty $\mathcal{B}\subseteq \mathcal{A}$ and for all $B\in \mathcal{B}$ there exists a $\subseteq$-maximal element $B^*$ of $\mathcal{B}$ such that $B\subseteq B^*$.   \begin{flushright}
	\kare
	\end{flushright}
\end{definition}

The following definition is meaningful because any Noetherian family $\mathcal{A}$ can be written of the form $\mathcal{A}=\bigcup_{i<\alpha} \mathcal{A}_i$ where $\mathcal{A}_i$  consists of
$\subseteq$-maximal elements of  $\mathcal{A}-\bigcup_{j<i} \mathcal{A}_i$ for all $i<\alpha$.

\begin{definition} \label{Definition Rank of Noetherian families} \cite{onal2020strategies}
	Let  $\mathcal{A}$ be a Noetherian family. We take any ordinal $i$ and  define $\mathcal{A}_i$ as the family  of $\subseteq$-maximal elements of $\mathcal{A}-\bigcup_{j<i} \mathcal{A}_j$ i.e., $\mathcal{A}_i=\{A\in \mathcal{A}-\bigcup_{j<i} \mathcal{A}_j : \forall B\in \mathcal{A}-\bigcup_{j<i} \mathcal{A}_j(A\subseteq B\rightarrow B=A) \}$, where $\mathcal{A}_j$ is already defined  for all $j<i$ in the same way.  
	Say the minimum element of the class   $\{\alpha:\mathcal{A}=\bigcup_{i<\alpha} \mathcal{A}_i\}$ of ordinals is $\alpha^*$ and so,  we have $\mathcal{A}=\bigcup_{i<\alpha^*}\mathcal{A}_i$. (Note that the family $\{\mathcal{A}_i:i<\alpha^*\}$ is disjoint and  $\mathcal{A}_i\neq \emptyset$ for all $i<\alpha^*$) Then, we define the following.
	
	1) The ordinal $\alpha^*$ is called \textbf{rank of $\mathcal{A}$} and denoted by \textbf{rank$(\mathcal{A})$}.
	
	2) What we mean with \textbf{
		Noetherian union of $\mathcal{A}$} is	$\bigcup_{i<rank(\mathcal{A})}\mathcal{A}_i$.
	
	3) We define a  function, \textbf{ rank function of  $\mathcal{A}$},   $r$ from $\mathcal{A}$ to $rank(\mathcal{A})$ such that $r(A)=i$ iff $A\in  \mathcal{A}_i$ for every $A\in \mathcal{A}$. The ordinal  $r(A)$ is called \textbf{level (or rank) of A in  $\mathcal{A}$}.  \begin{flushright}
		\kare
	\end{flushright}
	
\end{definition}

Let $\mathcal{A}$  be a Noetherian family and $A,B\in \mathcal{A}$.
Note that if $A\subsetneqq B$, then $r(B)<r(A)$.

\begin{lemma}\label{L: N}
	Let $\mathcal{A}$ be a Noetherian family. For any $A\in \mathcal{A}$, if  $r(A)=i>0$, 
then for each  $0\leq j<i$ there exists an $A_j\in \mathcal{A}$ such that  $r(A_j)=j$  and $A\subsetneqq A_j$. 

\end{lemma}

\begin{proof}
 Take any $j$ where $0\leq j<i$. So, $A\in \mathcal{A}-\bigcup_{m<i}\mathcal{A}_m\subseteq\mathcal{A}-\bigcup_{m<j}\mathcal{A}_m$. Thus, $A\in \mathcal{A}-\bigcup_{m<j}\mathcal{A}_m$.  Because $\mathcal{A}$ is Noetherian, there exists a $\subseteq$-maximal element $A_j$ of $\mathcal{A}-\bigcup_{m<j}\mathcal{A}_m$ such that $A\subseteq A_j$. Thus, $A_j\in \mathcal{A}_j$ and $r(A_j)=j$. Because $A\in \mathcal{A}_i$ and $\mathcal{A}_i\cap \mathcal{A}_j=\emptyset$, $A\neq A_j$. Therefore, $r(A_j)=j$  and $A\subsetneqq A_j$. 

\end{proof}

Let $\mathcal{A}$  be a Noetherian family. 
Take any $A,B\in \mathcal{A}$. Suppose  $A\subsetneqq B$. Then, we know that $r(B)<r(A)$. Suppose $r(B)+1<r(A)$. For any ordinal $\alpha$ where $r(B)<\alpha<r(A)$. It is possible that there is no $C_\alpha$ such that $A\subsetneqq C_{\alpha} \subsetneqq B$ and $r(C_\alpha)=\alpha$. Here is an easy example.

\begin{example}
 $\mathcal{A}=\{\{0,1,2\},\{1,2,3\},\{0,1\},\{1\}\}$. Say $A=\{1\}$ and $B=\{1,2,3\}$. Then, $r(A)=2>0=r(B)$ there is no $C\in \mathcal{A}$ such that  $r(C)=1$ and $A\subsetneqq C\subsetneqq B$.
\begin{flushright}
	\kare
\end{flushright}	
\end{example}

 \begin{remark}	
	
Let $\mathcal{A}$  be a Noetherian family, $\theta<\gamma <rank(\mathcal{A})$
	 ordinals and $A\in \mathcal{A}$ with $r(A)=\gamma$. For any given  finite sequence of ordinals $(\alpha_i)$ such that $\alpha_0=\gamma> \alpha_1>\alpha_2>\ldots>\alpha_n=\theta$ where $n\in \omega-\{0\}$, by repeatedly applying  Lemma \ref{L: N},  we can find $A_1, A_2, \ldots ,A_n\in \mathcal{A}$ such that  $A=A_0\subsetneqq A_1\subsetneqq A_2\subsetneqq\ldots\subsetneqq A_n$ and $r(A_j)=\alpha_j$ where $0\leq j\leq n$. However, form the example above,   if  $B\in \mathcal{A}$ with $r(B)=\theta$ is also  given at first, then  we may not find them. 
\begin{flushright}
	\kare
\end{flushright}
 	
\end{remark}

\begin{theorem}\label{t:2}

		Let $\mathcal{A}$ be a Noetherian family.
		 If  $\emptyset\neq \mathcal{B}\subseteq \mathcal{A}$, then $\mathcal{B}$ is Noetherian,   $rank(\mathcal{B})\leq rank(\mathcal{A})$ and for any $B\in \mathcal{B}$, the level of $B$ in $\mathcal{B}$ does not exceed the  level of $B$ in $\mathcal{A}$.

\end{theorem}

\begin{proof}
 Clearly, $\mathcal{B}$ is Noetherian. 
Say Noetherian union of $\mathcal{B}$ and $\mathcal{A}$  are  $\bigcup_{j<rank(\mathcal{B})} \mathcal{B}_j$ and  $\bigcup_{i<rank(\mathcal{A})} \mathcal{A}_i$, respectively.

\textbf{Claim.}  For all $j<rank(\mathcal{B})$ and for all $B\in \mathcal{B}_j$  if $B\in \mathcal{A}_i$ then $j\leq i$.

\textbf{Proof of the Claim.} We use transfinite induction. The case $j=0$ is obvious. Suppose  that $j<rank(\mathcal{B})$ and the claim is true  for any $k$ where $0\leq k< j$. Now, for the case $j$, take any $B\in \mathcal{B}_j$. Say  $B\in \mathcal{A}_i$.  Assume the opposite, say $i<j$. Then, from Lemma \ref{L: N}, there is $B^*\in  \mathcal{B}_i$ such that $B\subsetneqq B^*$. Say $B^*\in  \mathcal{A}_{i^*}$, so, from  the induction hypothesis,  $i\leq i^*$. Because   $B\subsetneqq B^*$, the level of $B$  in $ \mathcal{A}$, which is $i$, must be greater than  level of $B^*$  in $ \mathcal{A}$, which is $i^*$. So, $i>i^*$.  But  $i\leq i^*$. This is a contradiction. \kare

Assume  $rank(\mathcal{B})> rank(\mathcal{A})=\alpha$. Then, $\mathcal{B}_\alpha\neq \emptyset$. Take any $B\in \mathcal{B}_\alpha$ and  say $B\in \mathcal{A}_i$. So, $\mathcal{A}_i\neq \emptyset$ and from the claim, $rank(\mathcal{A})=\alpha\leq i$. This  contradicts  definition of $rank(\mathcal{A})$.

From the claim,  for any $B\in \mathcal{B}$, the level of $B$ in $\mathcal{B}$ does not exceed the  level of $B$ in $\mathcal{A}$.

\end{proof}

Let us be given a Noetherian base or  $\pi$-base   $\mathcal{B}$ of a space $X$. Can we reduce the rank? More clearly, can we find any base or $\pi$-base $\mathcal{B}'\subseteq \mathcal{B}$ such that $rank(\mathcal{B}')\leq |rank(\mathcal{B})|$ where $rank(\mathcal{B})$ is not a cardinal?  We give  a partial answer for this question in the following proposition and ask other part of it in  Section \ref{Ack sorular}.

\begin{theorem}\label{Prop: } If $\mathcal{B}$ is a Noetherian  $\pi$-base of a space $X$, then there exists a Noetherian $\pi$-base $\mathcal{B}'\subseteq \mathcal{B}$ such that $rank(\mathcal{B}')\leq \kappa$ where $|rank(\mathcal{B})|=\kappa$.
\end{theorem}

\begin{proof} Let Noetherian union  of $\mathcal{B}$ be $\bigcup_{i<rank(\mathcal{B})}\mathcal{B}_i$ and let $f$ be a bijection from $\kappa$ to $rank(\mathcal{B})$.

By using transfinite induction, we will  define some sets.  For $i=0$, say  $\mathcal{C}_0=\mathcal{B}_{f(0)}$. For any $i< \kappa-\{0\}$, assume that for every $j<i$, $\mathcal{C}_j$ is defined. After that,  define  $\mathcal{C}_i=\{B\in \mathcal{B}_{f(i)}:\forall j<i ( W\in \mathcal{C}_j\rightarrow W\nsubseteq B)\}$.

Now, define $\mathcal{B}'=\bigcup_{i< \kappa}\mathcal{C}_i$. 

Then $\mathcal{B}'$ is Noetherian because it is a subfamily of $\mathcal{B}$ and $\mathcal{B}$  is Noetherian. 

To see $\mathcal{B}'$ is a $\pi$-base take any $U\in \mathcal{B}$, then there exists an $i\in \kappa$ such that $U\in \mathcal{B}_{f(i)}$. If $U\in  \mathcal{C}_{i}$, then $U\in \mathcal{B}'$. If not, then there exists a $j<i$ and $W\in \mathcal{C}_j$ such that $W\subseteq U$. Thus $W\in \mathcal{B}'$. Therefore, $\mathcal{B}'$ is a $\pi$-base of $X$.

Now, we   see that $rank(\mathcal{B}')\leq \kappa$. Say Noetherian union of $\mathcal{B}'$ is $\bigcup_{i<rank(\mathcal{B}')}\mathcal{B}'_i$.  

\textbf{Claim.} For any $i\in \kappa$, $\bigcup_{j\leq i}\mathcal{C}_j\subseteq\bigcup_{j\leq i} \mathcal{B}'_j$.

\textbf{Proof of the Claim.} We use  transfinite induction. For  $i\in \kappa$, suppose it is true that  $\bigcup_{j\leq k}\mathcal{C}_j\subseteq\bigcup_{j\leq k} \mathcal{B}'_j$ for all $k<i$.

Take any $U\in \bigcup_{j\leq i}\mathcal{C}_i$ and assume $U\notin \bigcup_{j\leq i} \mathcal{B}'_j$. From the induction hypothesis and definition of $\mathcal{C}_i$,  $U\in \mathcal{C}_{i}\subseteq \mathcal{B}_{f(i)}$. Because $U\in \mathcal{B}'$ and  $U\notin \bigcup_{j\leq i} \mathcal{B}'_j$, there exists  a  $\mathcal{B}'_m$ such that $i<m<rank(\mathcal{B}')$ and $U\in \mathcal{B}'_m$. From Lemma \ref{L: N}, 
there exists a $D\in \mathcal{B}'_i$ such that $U\subsetneqq D$. Since $D\in \mathcal{B}$,  there exists $h\in \kappa$ such that $D\in \mathcal{C}_h\subseteq  \mathcal{B}_{f(h)}$. 
Since $U\subsetneqq D$, $D\in \mathcal{B}_{f(h)}$ and $U\in \mathcal{B}_{f(i)}$, we get $h\neq i$  and $f(h)<f(i)$. So, $h<i$ or $i<h$. Suppose  $h<i$. Because $D\in \mathcal{C}_h$, from induction hypothesis, $D\in \mathcal{C}_h \subseteq\bigcup_{j\leq h} \mathcal{B}'_j$. Therefore,  $D\in (\bigcup_{j\leq h} \mathcal{B}'_j)\cap \mathcal{B}'_i$. But $\{\mathcal{B}'_i:i<rank(\mathcal{B}')\}$ is disjoint. This is a contradiction. Suppose $i<h$. Then, from the definition of  $\mathcal{C}_h$, $D\notin \mathcal{C}_h$ because   $U\in\mathcal{C}_i$ and $U\subseteq D$. But $D\in \mathcal{C}_h$, this is a contradiction.  \kare
 
 Now suppose that $rank(\mathcal{B}')>\kappa$. Then, $\mathcal{B}'_\kappa\neq \emptyset$. Take any $B\in \mathcal{B}'_\kappa$. 
 Then there exists an $i\in rank(\mathcal{B})$ such that $B\in \mathcal{B}_i$. Because $f$ is a bijection,  
 there exists a  $j<\kappa$ such that  $f(j)=i$ and $B\in \mathcal{C}_{j}\subseteq \mathcal{B}_{f(j)}=\mathcal{B}_{i}$. Then, 
 from the claim, $B\in \mathcal{C}_{j}\subseteq \bigcup_{k\leq j}\mathcal{B}'_j$. So,  $B\in (\bigcup_{k\leq j} \mathcal{B}'_j)\cap \mathcal{B}'_\kappa$. But $\{\mathcal{B}'_i:i<rank(\mathcal{B}')\}$ is disjoint. This is a contradiction.

\end{proof}

\begin{remark} \label{C, B yildiz}
Note that being topological space or $\pi$-base are not essential in the proof above. Hence, the proof implies  that  for any given    Noetherian family $\mathcal{A}$, there exists a Noetherian sub-family $\mathcal{A}'\subseteq \mathcal{A}$ such that $rank(\mathcal{A}')\leq |rank(\mathcal{A})|$ and for any $A\in \mathcal{A}$ there is an $A'\in \mathcal{A}'$ with the condition that $A'\subseteq A$. \begin{flushright}
	 \kare
\end{flushright}
\end{remark}

We give the following definition which is  important for this article. Although it is clear from the context, it is useful to repeat that in this paper, especially, in the following definition and related results (as we mentioned before) if we write a set $A$ as $A=\{a_i:i\in \alpha\}$ where $\alpha$ is a  cardinal  or ordinal number, then we  suppose that  this matching is one to one, i.e, $a_i\neq a_j$ if $i\neq j$.

\begin{definition}\label{tanm}
	Let  $\mathcal{A}$ be a Noetherian family and let Noetherian union of $\mathcal{A}$ be	$\bigcup_{i<\alpha} \mathcal{A}_i$. What we mean with \textbf{an $\alpha\times \beta$ Noetherian table of $\mathcal{A}$} or shortly \textbf{a Noetherian table of $\mathcal{A}$}, which is denoted by $[\mathcal{A}]$, is the following. 
	
	\[\begin{bmatrix}
		A^0_0&A^1_0&A^2_0&\dots  \\
		A^0_1&A^1_1&A^2_1&\dots\\
		A^0_2&A^1_2&A^2_2&\dots\\
		\dots&\dots&\dots& \dots \\
		
	\end{bmatrix}\]
	
	There are $\alpha$, which is rank of $\mathcal{A}$, many rows in this Noetherian table and for each $i\in \alpha$ there  exists an ordinal $\theta_i$ such that $\mathcal{A}_i=\{A^j_i\in \mathcal{A}_i:j \in\theta_i\}$.  The  $i$-th row is the tuple $(A^0_i,A^1_i,A^2_i,\ldots)$ consisting of all elements of   $\mathcal{A}_i$.  $\beta=\min\{\xi:\xi$ is an ordinal and $\theta_i\leq  \xi$ for all $i\in \alpha \}$.

	For any $A^j_i$ in this Noetherian table, note that    $r(A^j_i)=i$, the ordinal $j$ is called  \textbf{ order of $A^j_i$} in $\mathcal{A}_i$ of $[\mathcal{A}]$ and it is written: $o(A^j_i)=j$.\begin{flushright}
		\kare
	\end{flushright}

\end{definition}

Note that $r(A)$ is same in any Noetherian table,  but $o(A)$ can be different from Noetherian table to Noetherian table for any $A$ in a Noetherian family $\mathcal{A}$. So, if we need to fix $o(A)$, we have to fix a Noetherian table $[\mathcal{A}]$ for a Noetherian family $\mathcal{A}$.  Note that for any Noetherian table $[\mathcal{A}]$, $o(A)$ cannot reach the cardinal number $|\mathcal{A}|^+$, i.e.,  $o(A)<|\mathcal{A}|^+$. But it can be that $|\mathcal{A}|<o(A)<|\mathcal{A}|^+$.

\begin{remark}\label{bunada}
	Any Noetherian family  $\mathcal{A}$ has an  $\alpha\times\beta$ Noetherian table   such that  $rank(\mathcal{A})=\alpha$ and  where Noetherian union of $\mathcal{A}$ is  $\bigcup_{i<\alpha} \mathcal{A}_i$, $\beta=\min\{\xi:\xi$ is an ordinal and $|\mathcal{A}_i|\leq\xi$ for all $i \in \alpha\}$. Because, we can write $\mathcal{A}_i=\{A^j_i:j\in |\mathcal{A}_i|\}$  and, with Definition \ref{tanm},  $\beta=\min\{\xi:\xi$ is an ordinal and $|\mathcal{A}_i|\leq\xi$ for all $i \in \alpha\}$. \begin{flushright}
		\kare
	\end{flushright}
\end{remark}

\begin{lemma}\label{sona}
Let $\kappa$ be an infinite cardinal, $\mathcal{A}$  a Noetherian family and $[\mathcal{A}]$ an $\alpha\times\beta$ Noetherian table of  $\mathcal{A}$. If $\alpha,\beta\leq \kappa$, then $|\mathcal{A}|\leq \kappa$. 	
\end{lemma}
\begin{proof}
	Let $\bigcup_{i<\alpha}\mathcal{A}_i$ be the Noetherian union of $\mathcal{A}$. Then (from Definition \ref{tanm}), $|\mathcal{A}_i|\leq \beta$ for every $i\in \alpha$. So,  for every $i\in \alpha$, $|\mathcal{A}_i|\leq |\beta|$. Thus, $\sup\{|\mathcal{A}_i|:i\in \alpha\}\leq \beta$. (By using this fact: $|\bigcup_{i<\alpha}\mathcal{A}_i|\leq |\alpha| \cdot\sup\{|\mathcal{A}_i|:i\in \alpha\}$, Theorem 1.5.14 in \cite{holz2010introduction}, we have the following:)
	$|\mathcal{A}|=|\bigcup_{i<\alpha}\mathcal{A}_i|\leq|\alpha| \cdot\sup\{|\mathcal{A}_i|:i\in \alpha\}\leq |\alpha|\cdot|\beta|=\max\{|\alpha|,|\beta|\}\leq \kappa$.
\end{proof}

Let $\mathcal{A}$ be a Noetherian family and $[\mathcal{A}]$ an $\alpha\times\beta$ Noetherian table of  $\mathcal{A}$. Being   $|\mathcal{A}|\leq \kappa$ does not imply that $\alpha,\beta\leq \kappa$. Here is an example.

\begin{example}
 $\mathbb{R}$ stands for reel line and $[a,b]$ stands for the intervals of $\mathbb{R}$ in this example. 

Suppose $\mathcal{C}$ is a family which has the conditions:  

(1) $\emptyset\notin\mathcal{C}$ and  $|\mathcal{C}|=\omega$, 

(2) for any $A,B\in \mathcal{C}$ if $A\neq B$, then $A\nsubseteq B$ and $B\nsubseteq A$, 

(3) for any  $A\in \mathcal{C}$, $A\cap \mathbb{R}=\emptyset$.  

Suppose $\mathcal{D}=\{[1-\frac{1}{n+1},4+\frac{1}{n+1}]:n\in \omega\}\cup \{[2-\frac{1}{n+1},3+\frac{1}{n+1}]:n\in \omega\}$.

Now, say $\mathcal{A}=\mathcal{C}\cup\mathcal{D}$. \\Then, clearly $\mathcal{A}$ is Noetherian and $rank(\mathcal{A})=\omega+\omega$. Say Noetherian union of $\mathcal{A}$ is $\bigcup_{i<{\omega+\omega}}\mathcal{A}_i$. Thus, $\mathcal{A}_0=\{[0,5]\}\cup \mathcal{C}$ and $\mathcal{A}_i=\{[1-\frac{1}{i+1},4+\frac{1}{i+1}]\}$ where  $0<i<\omega$.  $\mathcal{A}_\omega=\{[1,4]\}$  and $\mathcal{A}_i=\{[2-\frac{1}{i+1},3+\frac{1}{i+1}]\}$ where  $\omega<\omega+i<\omega+\omega$.
 Fix any ordinal $\beta$ where $\omega <\beta< \omega_1$. Then we can write $\mathcal{A}_0=\{A_0^j:j\in \beta\}$ and we can define an $(\omega+\omega)\times \beta$ Noetherian table 
 $$[\mathcal{A}]=\begin{bmatrix}
	A_0^0&A_0^1&A_0^2&A_0^3&\ldots\\A_0^1\\A_0^2\\ \ldots	
\end{bmatrix}.$$ 
Thus, $|\mathcal{A}|= \omega$, $\omega <\beta$ and $\omega <\omega+\omega$ where $\mathcal{A}$ has an  $(\omega+\omega)\times \beta$ Noetherian table.
\begin{flushright}
	\kare
\end{flushright}

\end{example}

\begin{theorem}\label{yeter}
Let $\kappa$ be an infinite cardinal, $\mathcal{A}$  a Noetherian family having an $\alpha\times\beta$ Noetherian table where $\alpha,\beta\leq \kappa$. Then, for any nonempty $\mathcal{B}\subseteq \mathcal{A}$  there is a $\theta\times \delta$ Noetherian table of  $\mathcal{B}$ such that $\theta \leq \alpha\leq\kappa$ and  $\delta\leq \kappa$. 	
\end{theorem}

\begin{proof} From Theorem \ref{t:2},  $\mathcal{B}$ is Noetherian and  $rank(\mathcal{B})=\theta \leq \alpha=rank(\mathcal{A})\leq\kappa$.  From Lemma \ref{sona}, $|\mathcal{A}|\leq \kappa$. So, $|\mathcal{B}|\leq \kappa$, because $\mathcal{B}\subseteq \mathcal{A}$. Let $\bigcup_{i<\theta}\mathcal{B}_i$ be the Noetherian union of $\mathcal{B}$. Then for every $i<\theta$, $|\mathcal{B}_i|\leq \kappa$ because $\mathcal{B}_i\subseteq \mathcal{B}$ and $|\mathcal{B}|\leq \kappa$.
 If we say $\delta=\min\{\xi:\xi$ is an ordinal and $|\mathcal{B}_i|\leq\xi$ for all $i \in \theta\}$, then $\delta\leq \kappa$ and 
 from Remark \ref{bunada}, $\mathcal{B}$ has an  $\theta\times\delta$ Noetherian table. 
\end{proof}
We give the following remark about the theorem above.
\begin{remark}\label{bitmedi}
Let $\mathcal{A}$ be a Noetherian family having an $\alpha\times\beta$ Noetherian table. Then, for any nonempty $\mathcal{B}\subseteq \mathcal{A}$, can we say that   there is a $\theta\times \delta$ Noetherian table of  $\mathcal{B}$ such that $\theta \leq \alpha$ and  $\delta\leq \beta$?

From Theorem \ref{yeter}, we have $\theta \leq \alpha$. But it can be that $\delta> \beta$. The following example is about this. \begin{flushright}
	\kare
\end{flushright} 	
\end{remark}

\begin{example}\label{gitti}
Let $\kappa$ be an infinite cardinal. For all $i\in \kappa$, define these sets $a_i=\{j\in \kappa:i\leq j<\kappa\}$ and $b_i=\{i\}$. Thus, $a_0\supsetneqq a_1 \supsetneqq \ldots \supsetneqq a_i\supsetneqq a_{i+1}\supsetneqq\ldots$. In addition,  $\{b_i:i\in \kappa\}$  is a disjoint family and  $b_i\subsetneqq a_i$ for all $i\in \kappa$. Now, say $\mathcal{A}=\{a_i:i\in \kappa\}\cup\{b_i:i\in \kappa\}$. Then it is easy to check that    $\mathcal{A}$ has the following three properties.

(1) $\mathcal{A}$ is Noetherian.

(2) Noetherian union of $\mathcal{A}$ is $\bigcup_{i< \kappa}\mathcal{A}_i$ such that for all $i<\kappa$,   $\mathcal{A}_{i+1}=\{a_{i+1}, b_i\}$ and if $i$ is a limit ordinal or $i=0$, then $\mathcal{A}_i=\{a_i\}$.

(3) $\mathcal{A}$ has a $\kappa\times 2$ Noetherian table.\\ Say $\mathcal{B}=\{b_i:i\in \kappa\}$. Then it is easy to check that    $\mathcal{B}$ has the following  three properties.

(1) $\mathcal{B}$ is Noetherian.

(2) Noetherian union of $\mathcal{B}$ is $\bigcup_{i< 1}\mathcal{B}_i=\mathcal{B}_0$ such that $\mathcal{B}_0=\{b_i:i\in \kappa\}$.

(3) $\mathcal{B}$ has a $1\times \kappa$ Noetherian table. \\ Thus, $\mathcal{A}$ has a 
$\kappa\times 2$ Noetherian table, $\mathcal{B}\subseteq \mathcal{A}$ and  $\mathcal{B}$ has an $1\times \kappa$ Noetherian table. Furthermore, $\mathcal{B}$ has no $\theta\times \delta$ Noetherian table if $\delta \leq 2$.

\begin{flushright}
	\kare
\end{flushright} 	
\end{example}

It is natural to ask that does every topological space  have a Noetherian base or $\pi$-base? It is known that
every metric space $X$ has a Noetherian base $\mathcal{B}$ with $rank(\mathcal{B})\leq \omega$ (\cite{onal2020strategies}, p. 6). But in general, we cannot say that every topological space has a Noetherian base, because there are some spaces without Noetherian bases \cite{tamariz1988example}. For $\pi$-bases, answer of this question is affirmative because  we show that every topological space has a Noetherian $\pi$-base in the following theorem.

\begin{theorem}\label{Theorem: pi Base varsa Noetherian pi Base var}
	If $\mathcal{B}$ is a $\pi$-base of a space $X$ and $|\mathcal{B}|=\kappa$. Then, there exist a Noetherian $\pi$-base $\mathcal{B}'\subseteq \mathcal{B}$ such that

	(i) $rank(\mathcal{B'})\leq \kappa$ and $\mathcal{B'}$ has a $rank(\mathcal{B'})\times \beta$  Noetherian table  where  $\beta\leq \kappa$,

	(ii) $|\{B\in \mathcal{B'}: W\subseteq B \}|<\kappa$ for all $W\in \mathcal{B'}$.
  \end{theorem}

\begin{proof}
Say $\mathcal{B}=\{B_i:i\in \kappa\}$.

Choosing procedure: For $i=0$, $B_0$ is chosen. For any $i< \kappa-\{0\}$, assume for every $j< i$ this procedure is completed, after that, if there exists a $j<i$ such that $B_j\subsetneqq B_i$ and $B_j$ is chosen, then  $B_i$ is not chosen, otherwise $B_i$ is chosen.

Now, define $\mathcal{C}=\{B\in \mathcal{B}:$ B is chosen$\}$. 

It is easy to see that $\mathcal{C}$ is a $\pi$-base. To prove  $\mathcal{C}$ is a Noetherian, fix any increasing  sequence  $A_0\subsetneqq A_1\subsetneqq A_2 \subsetneqq\ldots $ of elements of  $\mathcal{C}$. Then, there exists $i_k\in \kappa$ such that $A_k=B_{i_k}\in \mathcal{B}$.  Now, if we suppose $j_0<j_1$, it means that there exists $j_0<j_1$ such that $B_{j_0}\subsetneqq B_{j_1}$ and $B_{j_0}$ is chosen. It contradicts   the choosing procedure because $ B_{j_1}$ is also chosen. Thus, we get $j_0>j_1$. With the same argument, we get   $j_1>j_2$ and $j_2>j_3$, etc. Because there is no strictly decreasing sequence of ordinals, the  increasing  sequence  $A_0=B_{j_0}\subsetneqq A_1=B_{j_1}\subsetneqq A_2=B_{j_2}\subsetneqq\ldots$ is finite.
Therefore, $\mathcal{C}$ is a Noetherian $\pi$-base. 
$|rank(\mathcal{C})|\leq\kappa$ because $|\mathcal{B}|\leq\kappa$ and $\mathcal{C}\subseteq \mathcal{B}$.  
 From Theorem \ref{Prop: },  there exists a  Noetherian $\pi$-base $\mathcal{B'}\subseteq\mathcal{C}$  such that 
 $rank(\mathcal{B'})\leq \kappa$.
 
 (i) Thus, 
 $rank(\mathcal{B'})\leq \kappa$. Let Noetherian uninon of $\mathcal{B'}$ be $\bigcup_{i\in {rank(\mathcal{B'})}}=\mathcal{B}'_i$ and $\beta=\min\{\xi:\xi$ is an ordinal and $|\mathcal{B}'_i|\leq \xi$ for all $i\in  rank(\mathcal{B'})\}$.  Because $|\mathcal{B'}|\leq\kappa$, $|\mathcal{B}'_i|\leq\kappa$. So,  $\beta\leq \kappa$. Thus,  
  $\mathcal{B'}$ has a $rank(\mathcal{B'})\times \beta$  Noetherian table where  $rank(\mathcal{B'}),\beta\leq \kappa$.

(ii)  Take any $W\in \mathcal{B'}$. Because $\mathcal{B'}\subseteq \mathcal{C}\subseteq \mathcal{B}$ and  $\mathcal{B}=\{B_i:i\in \kappa\}$ and the choosing procedure, there exists an  $i_0\in \kappa$ such that $B_{i_0}=W$ and $B_{i_0}$ was chosen. Because of the choosing procedure, for any $B\in \mathcal{B'}$, if $B_{i_0}\subseteq B$, then $B\in \{B_i\in\mathcal{B}: i\le{i_0}\text{ and $B_i$ was chosen}\}$. Thus, $\{B\in \mathcal{B'}: W\subseteq B \}\subseteq \{B_i\in\mathcal{B}: i\le{i_0}\text{ and $B_i$ was chosen}\}$.  Because $|\{B_i\in\mathcal{B}: i\le{i_0}\text{ and $B_i$ was chosen}\}|\leq|i_0|<\kappa$, we have $|\{B\in \mathcal{B'}: W\subseteq B \}|<\kappa$.

\end{proof}

\begin{corollary}
Let $X$ be space and $\pi w(X)\leq\kappa$. For any $\pi$-base  $\mathcal{B}$  of  $X$, there exists a Noetherian $\pi$-base $\mathcal{B}'\subseteq \mathcal{B}$ such that

(i) $rank(\mathcal{B'})\leq \kappa$ and $\mathcal{B'}$ has a $rank(\mathcal{B'})\times \beta$  Noetherian table  where  $\beta\leq \kappa$,

(ii) $|\{B\in \mathcal{B'}: W\subseteq B \}|<\kappa$ for all $W\in \mathcal{B'}$.	
\end{corollary}
\begin{proof}
Suppose $\mathcal{C}$  is a $\pi$-base  for $X$ with $|\mathcal{C}|=\mu\leq \kappa$. Let $\mathcal{C}=\{C_i:i\in \mu\}.$ Thus, for any $i\in \mu$, we can choose a $B_i\in \mathcal{B}$ such that $B_i\subseteq C_i$ and define  $\mathcal{B}^*=\{B_i:i\in \mu\}$. Therefore, $\mathcal{B}^*$ is a $\pi$-base for $X$ and $|\mathcal{B}^*|\leq \kappa.$ Thus, from Theorem \ref{Theorem: pi Base varsa Noetherian pi Base var},  there exist a Noetherian $\pi$-base $\mathcal{B}'\subseteq \mathcal{B}^*$ which has the conditi
ons (i) and (ii). 
 \end{proof}

The following corollaries are direct   results of Theorem \ref{Theorem: pi Base varsa Noetherian pi Base var}. So, we omitted their proofs.  

\begin{corollary}\label{k}
Let $\mathcal{B}$ be a $\pi$-base  of a space $X$. Then, there exists a Noetherian $\pi$-base $\mathcal{B'}$ of $X$ such that $\mathcal{B'}\subseteq \mathcal{B}$. \begin{flushright}
	\kare
\end{flushright}
\end{corollary}

\begin{corollary}\label{Cor:  sheme}
	Let $X$ be a space.  $\pi w(X)\leq\kappa$ if and only if  $X$ has a Noetherian $\pi$-base $\mathcal{B}$ such that $|\mathcal{B}|\leq \kappa$ and $\mathcal{B}$ has an $\alpha\times \beta$  Noetherian table where  $\alpha, \beta\leq \kappa$. \begin{flushright}
		\kare
	\end{flushright}
\end{corollary}

Corollary \ref{k} is including the following corollary. But we think that it would be more catchy to state Corollary \ref{Cor: Every space has noeth base} separately.

\begin{corollary}\label{Cor: Every space has noeth base}
	Every topological space has a Noetherian $\pi$-base. \begin{flushright}
		\kare
	\end{flushright}
\end{corollary}

\section{From winning strategy to 2-tactic for \2 in BM(X)}

In this section, we study on some topological spaces which have some Noetherian $\pi$-bases to get  \2's  2-tactic from a winning strategy. We have done most of the things mentioned in the introduction in this section. 

We first give some results and define the  \textit{purification} of open families. After that, we give  Theorem \ref{Teorem winning ten 2-tactige ana sonuc 2}. This theorem is one of our main results in this paper and it also includes the Galvin's theorem.

The statement "$O$ has a cellular family of cardinality $\omega$" means that  there exists a disjoint family $\mathcal{A}\subseteq \tau^*(O)$ such that $|\mathcal{A}|=\omega$.

Let's start with the following lemma that is most likely known.
\begin{lemma}\label{l:1} Let $X$ be a space and $U\in \tau^*(X)$. 
	For every $n\in \omega$ if  $U$  has a cellular family of cardinality  $n$, then $U$  has a cellular family of cardinality $\geq\omega$.
	
\end{lemma}

\begin{proof}
	
	\textbf{Claim.} Let $Y$ be a space and suppose for every $n\in \omega$,   $Y$  has a cellular family of cardinality  $n$. 
	
	(1) There exists a  maximal cellular family  $\mathcal{A}\subseteq \tau^*(Y)$ such that $2\leq |\mathcal{A}|$ and $\mathcal{A}$ is finite.
	
	(2) For any maximal cellular family  $\mathcal{A}\subseteq \tau^*(Y)$ if $\mathcal{A}$ is  finite, then there exists an $A\in \mathcal{A}$ such that $A$ also has that property, i.e., for every $n\in \omega$,  $A$  has a cellular family of cardinality  $n$.
	
	\textbf{Prof of the Claim.} (1) For $n=2$, $Y$  has a cellular family of cardinality  $2$, say this family is $\mathcal{A}_2$. If $\mathcal{A}_2$ is maximal, then say $\mathcal{A}=\mathcal{A}_2$. If not, define $W=Y-\overline{\cup\mathcal{A}}$. And say $\mathcal{A}=\mathcal{A}_2\cup\{W\}$.  (2) Suppose $\mathcal{A}=\{A_1,A_2,\ldots, A_m\}$ is finite and there is no $A_i\in \mathcal{A}$ which has that property. Then, for each $i\leq m$, there are $n_i\in \omega$ and  $\mathcal{A}_{n_i}$  cellular  family of $A_i$ such that $|\mathcal{A}_{n_i}|=n_i$ and if $\mathcal{C}$ is a  cellular  family of $A_i$, then $|\mathcal{C}|\leq n_i$. Say $\max\{n_i:i\leq m\}=n^*$. Then, from the hypothesis, there exists a 
	cellular family $\mathcal{D}\subseteq \tau^*(Y)$ 
	of $Y$ such that $|\mathcal{D}|=n^*.m+1$. Because $\mathcal{A}$ is maximal cellular, for each $D\in \mathcal{D}$ there is an $A_D\in \mathcal{A}$ such that $D\cap A_D\neq \emptyset$. Thus, it easy to check that there exists $A_{i^*} \in \mathcal{A}$
	such that $|\{D\in \mathcal{D}:D\cap A_{i^*}\neq \emptyset\}|\geq n^*+1>n_{i^*}$. Therefore $\{D\cap A_{i^*}\neq \emptyset:D\in \mathcal{D}\}\subseteq \tau^*(A_{i^*})$ is disjoint and $|\{D\cap A_{i^*}\neq \emptyset:D\in \mathcal{D}\}|>n_{i^*}$ but it contradicts   $|\mathcal{A}_{n_{i^*}}|=n_{i^*} \geq |\{D\cap A_{i^*}\neq \emptyset:D\in \mathcal{D}\}|$. \kare

	Let $U\in \tau^*(X)$ and suppose  
	for every $n\in \omega$, $U$  has a cellular family of cardinality  $n$. From the claim, we can say that there exists a finite maximal cellular  family $\mathcal{A}_0\subseteq \tau^{*}(U)$ of $U$ such that $|\mathcal{A}_0|\geq 2.$ and  there exist an  $A_0\in \mathcal{A}_0$ with the condition that  for every $n\in \omega$, $A_0$  has a cellular family of cardinality  $n$. Again from the claim, we can say that there exists a finite maximal cellular  family $\mathcal{A}_1\subseteq \tau^{*}(A_0)$ of $A_0$ such that $|\mathcal{A}_1|\geq 2$ and  there exist $A_1\in \mathcal{A}_1$ with the condition that  for every $n\in \omega$, $A_1$  has a cellular family of cardinality  $n$. Thus, inductively,  for each $i\in \omega-\{0\}$  there exists a finite maximal cellular  family $\mathcal{A}_i\subseteq \tau^{*}(A_{i-1})$ of $A_{i-1}$ such that $|\mathcal{A}_i|\geq 2.$ and  there exist $A_i\in \mathcal{A}_i$ with the condition that  for every $n\in \omega$, $A_i$  has a cellular family of cardinality  $n$. After that, for each $i\in \omega$ if we choose $B_i\in \mathcal{A}_i-\{A_i\}$, then $\{B_i:i\in \omega\}$ is the cellular  family of $U$ such that  $|\{B_i:i\in \omega\}|=\omega$.

\end{proof} 
The following proposition is most likely known.
\begin{proposition}\label{l:2} Let $X$ be a space and $U\in \tau^*(X)$.  For every disjoint family $\mathcal{A}\subseteq \tau^*(U)$, suppose that $|\mathcal{A}|<\omega$. Then, there exists an $n\in \omega$ such that for every disjoint family $\mathcal{A}\subseteq \tau^*(U)$,  $|\mathcal{A}|\leq n$.
	
\end{proposition}

\begin{proof}  This follows from Lemma \ref{l:1}.

\end{proof}
A space $X$ is called \textbf{Baire} if the countable intersection of dense open subsets of $X$ is dense in $X$. It is folklore that if \2 has a winning strategy on a space $X$, then $X$ is a Baire space (\cite{telgarsky1987topological}, p. 234). It is known that ever nonempty open subspace of a Baire space is also Baire. These facts are used in the following proof. 
\begin{theorem}\label{t:1}
	Let $X$ be a space, suppose \2 has a winning strategy in $BM(X)$ and there is a  $W\in \tau^*(X)$ with the condition  that there exists an $n\in \omega-\{0\}$ such that for every disjoint family $\mathcal{A}\subseteq \tau^*(W)$,  $|\mathcal{A}|\leq n$. In any play of  $BM(X)$ if \1 makes a move $U_i\supseteq W$, then \2 wins this play, even \2  has a 1-tactic after the move $U_i$ of \1. 
\end{theorem}

\begin{proof} Say $n'=\min\{n\in \omega-\{0\}:$ for every disjoint family $\mathcal{A}\subseteq \tau^*(W),|\mathcal{A}|\leq n \}$. Then, there exists a disjoint  family $\mathcal{A}'\subseteq \tau^*(W)$ such that $|\mathcal{A}'|= n'$ and for any  $A\in \mathcal{A}'$ if $\mathcal{C}\subseteq \tau^*(A)$ is nonempty disjoint, then $|\mathcal{C}|=1$. Now, fix any $A^*\in \mathcal{A}'$. Note that  $A^*$ is Baire space because $A^*$ is an open subset of $X$ and $X$ is Baire space as \2 has a winning strategy in $BM(X)$. Also note that
 any  nonempty open subset of $A^*$  is dense in $A^*$ (i.e., for any $V\in \tau^*(A^*)$, $A^*\subseteq \overline{V}$. Otherwise, $\mathcal{C}=\{V, A^*-\overline{V}\}\subseteq \tau^*(A^*)$ is disjoint and $|\mathcal{C}|=2$).

	Let $\sigma$ be the winning strategy for \2 in $BM(X)$ and suppose \1 makes the move $U_i\supseteq W$ somewhere in an arbitrary  play. \2 responds with the move $V_i=A^*$. After that, for $m>i$, where \1 makes a move $U_m$, then \2 responds with same move i.e., $V_m=U_m$ (Note that this is a 1-tactic for \2.). Thus, this play has formed like this:
	$$(U_0,V_0,\ldots, U_i,V_i={A^*}, U_{i+1},V_{i+1}=U_{i+1}, U_{i+2},V_{i+2}=U_{i+2},\ldots)$$

We know that $U_{i+j}$ nonempty open subset of $A^*$ and $U_{i+j}$ dense in $A^*$ where $1\leq j$. Because $A^*$ is Baire space (and $\{U_{i+j}:j\in \omega-{0}\}$ consists of dense open  subsets), $\emptyset \neq \bigcap_{j\in\omega-\{0\}}U_{i+j}= \bigcap_{j\in\omega}U_j$. It means that \2 wins this play. 
\end{proof}

Note that if $X$ is a $T_2$ space in the theorem above, then $W$ is finite and consists of isolated points. Therefore, the proof becomes much easier, as shown in the following observation. However, we do not assume any separation axioms on the space $X$. 

Additionally, in the above theorem,  $U_i\cap W\neq \emptyset$ is a sufficient condition instead of $U_i\cap W\neq \emptyset$. This follows from the fact that if the hypothesis holds for $W$, then it also holds for the nonempty $U_i\cap W$, i.e., $U_i\supseteq U_i\cap W$,  $U_i\cap W \in \tau^*(X)$ and $|\mathcal{A}|\leq n$ for 
 every disjoint family $\mathcal{A}\subseteq \tau^*(U_i\cap W)$.

\textbf{Observation}: Let  $X$ be a space with a $\pi$-base $\mathcal{B}$.  If $B$ has an isolated point for each  $B\in \mathcal{B}$, then 
\2 has  a 1-tactic in $BM(X)$. To see this for any nonempty open subset $U$ of $X$, take any isolated point $x_U\in U$ and define $\sigma(U)=\{x_U\}$. Obviously, $\sigma$ is a  1-tactic in $BM(X)$ for \2. Thus, this is also a 2-tactic. So, \2 has a 1-tactic in $BM(X)$ if $X$ is a scattered space.
 \begin{flushright}
	\kare
\end{flushright}

We give the following definition which is  necessary for this article.

\begin{definition}

		Let $X$ be a space and $\mathcal{A}\subseteq \tau^*(X)$. Then, we define and call \textbf{the purification of $\mathcal{A}$}, $pur( \mathcal{A})$, as follows:  $pur(\mathcal{A})=\{U\in \mathcal{A}: $ if $W \in \tau^*(U)$, then $W$ has a cellular family of cardinality $\omega\}$.
		
		Note that if $U\in pur(\mathcal{A})$, then $U$ has no isolated points.
		\begin{flushright}\kare	
		
	\end{flushright}
\end{definition}

\begin{corollary}\label{C:3} For any space $X$ the following are equivalent.
	
	(1) $pur(\tau^*(X))=\emptyset$;

	(2) For every $\pi$-base $\mathcal{B}$ of $X$,  $pur(\mathcal{B})=\emptyset$;
	
	(3) There exists a  $\pi$-base $\mathcal{B}$ of $X$ such that $pur(\mathcal{B})=\emptyset$.
	
\end{corollary}

\begin{proof}
This follows immediately from definition above.	
\end{proof}
\begin{remark}\label{r:g}
Let $X$  be a  space. Suppose \2 has a winning strategy in $BM(X)$ and suppose we want to prove the existence of \2's 2-tactic (even k-tactic or k-Markov strategy). Let $U_i$ be a move of \1 in any arbitrary play of $BM(X)$. Then,  without loss of generality, we can suppose that $U_i\in pur(\tau^*(X))$. Otherwise  \2 has a 1-tactic (so, 2-tactic,  k-tactic or k-Markov) after the move $U_i$.
\end{remark}
\begin{proof} Suppose $U_i\notin pur(\tau^*(X))$.
Then, there exists a $W\in \tau^*(U_i)$ such that  $W$ has no cellular family of cardinality $\omega$, i.e., $|\mathcal{A}|<\omega$ for any disjoint family $\mathcal{A}\subseteq \tau^*(W)$, then from Proposition \ref{l:2}, there exists an $n\in \omega$ such that for any disjoint family $\mathcal{A}\subseteq \tau^*(W)$,  $|\mathcal{A}|\leq n$. Thus, from Theorem  \ref{t:1}, \2 has a   1-tactic (so, 2-tactic or others) after the move $U_i$. 	
\end{proof}

Corollary \ref{C:3} and similar arguments as in the proof of Remark \ref{r:g} yield also the following  corollary.
\begin{corollary}\label{c:makale}

 Let $X$  be a  space. Suppose \2 has a winning strategy in $BM(X)$. For a $\pi$-base $\mathcal{B}$ of $X$ if $pur(\mathcal{B})=\emptyset$, then \2 has a 1-tactic (so, 2-tactic) in $BM(X)$. \begin{flushright}
	\kare
\end{flushright}
\end{corollary}

As we mentioned before, the following theorem is one of the main results of this article.  For the differences between  ($\star$) and ($\star\star$)-conditions see Remark \ref{rem}.
 
\begin{theorem}\label{Teorem winning ten 2-tactige ana sonuc 2} 
	Let $X$ be a
	space and  $\mathcal{B}$ a $\pi$-base of $X$ having the following ($\star$)-condition. If \2 has a winning strategy in $BM(X)$, then \2 has a  2-tactic in $BM(X)$. 
	
	($\star$) $pur(\mathcal{B})=\emptyset$  or  $pur(\mathcal{B})$ is   Noetherian and has a 
	Noetherian table $[pur(\mathcal{B})]$ such that
	for any $B\in [pur(\mathcal{B})]$
	there exists a disjoint family $\mathcal{A}_{B}\subseteq \tau^*(B)$ such that 
	$|\mathcal{A}_{B}|\geq \max\{|r(B)|,M_B\}$ where Noetherian union of $pur(\mathcal{B})$ is  $\bigcup_{\theta<\alpha}\mathcal{B}_\theta$ and $M_B=\sup\{|\{V\in  \mathcal{B}_\theta:B\subseteq V, o(V)\leq o(B)\}|:\theta\leq r(B)\}$. 
	
\end{theorem}

\textit{Remark.} This theorem is also true if ($\star$)-condition is replaced by the following ($\star\star$)-condition. Note that where $B\in pur(\mathcal{B})$, $r(B)$ and $o(B)$ depend on the Noetherian table $[pur(\mathcal{B})]$ in ($\star$)-condition while $r(B)$ and $o(B)$ depend on the Noetherian table $[\mathcal{B}]$ in ($\star\star$)-condition.

($\star\star$) $pur(\mathcal{B})=\emptyset$ or where $pur(\mathcal{B})\neq\emptyset$, $\mathcal{B}$ is  Noetherian and has a 
Noetherian table $[\mathcal{B}]$ such that
for  any  $B\in pur(\mathcal{B})$, suppose  $B=U^{o(B)}_{r(B)}\in [\mathcal{B}]$ then,
there exists a disjoint family $\mathcal{A}_{B}\subseteq \tau^*(B)$ such that 
$|\mathcal{A}_{B}|\geq \max\{|r(B)|,M_B\}$ where Noetherian union of $\mathcal{B}$ is  $\bigcup_{\theta<\alpha}\mathcal{B}_\theta$ and $M_B=\sup\{|\{V\in  \mathcal{B}_\theta:B\subseteq V, o(V)\leq o(B)\}|:\theta\leq r(B)\}$.

\begin{proof}  Suppose $pur(\mathcal{B})=\emptyset$. Then,  from Corollary \ref{c:makale} the proof is completed. 
	
	Suppose  $pur(\mathcal{B})\neq \emptyset$. Then, from Remark \ref{r:g}	without loss of generality, we may assume that in each  round of any play if $U_i$ is  \1's move  then $U_i\in pur(\tau^*(X))$.   
	
	Now put a well order on $\mathcal{B}$ and say $U'=\min \{B\in \mathcal{B}: B\subseteq U\}$
	where $U\in \tau^*(X)$. Note that $U'\in pur(\mathcal{B})$ if $U$ is a move of \1 (because $U\in pur(\tau^*(X))$).	
	
	Let  $[pur(\mathcal{B})]$ be an  $\alpha\times\beta$ Noetherian table with the  Noetherian union $\bigcup_{i<\alpha}\mathcal{B}_i$.
	
		To get the  following three definitions take  and fix any  $B\in pur(\mathcal{B})$.

	(1-Definition of $\widehat{B}$)  Say $h(B)=\min \{o(U)\in \beta: U\in pur(\mathcal{B}), U\subseteq B\}$  and $s(B)= \min \{s\in \alpha: U^{h{(B)}}_s\in [pur(\mathcal{B})], U^{h{(B)}}_s\subseteq B\}$. 
	Define  $\widehat{B}=U^{h{(B)}}_{s(B)}\in [pur(\mathcal{B})]$.

	(2-Definitions of $\mathcal{C}_{B,\theta}, m_{B,\theta}$) Take any  $\theta \leq r(B)$.  Say $|\{V\in  \mathcal{B}_\theta:B\subseteq V, o(V)\leq o(B)\}|=m_{B,\theta}$ and  $\mathcal{C}_{B,\theta}=\{V\in  \mathcal{B}_\theta:B\subseteq V, o(V)\leq o(B)\}$. Then we can rewrite  and fix $\mathcal{C}_{B,\theta}=\{V_l:l \in m_{B,\theta}\}$.  
	(So, $M_{B}=\sup\{m_{B,\theta}:\theta \leq r(B)\}$.)

	(3-Definition of $\mathcal{O}_{\widehat{B}}$) Because  $\widehat{B}$ has a cellular family of cardinality $\geq\omega$, by using hypothesis, we can fix a disjoint family $\mathcal{A}_{\widehat{B}}\subseteq \tau^*(\widehat{B})$ such that $|\mathcal{A}_{\widehat{B}}|\geq \max\{ \omega,|r(\widehat{B})|, M_{\widehat{B}}\}$. Thus, by fixing a one to one function $f$ from the set  $\omega\times (r(\widehat{B})+1) \times M_{\widehat{B}}$ to disjoint family 
	$\mathcal{A}_{\widehat{B}}$, we can enumerate the subset  $f[\omega\times (r(\widehat{B})+1) \times M_{\widehat{B}}]$ of $\mathcal{A}_{\widehat{B}}$ such 
	that $f[\omega\times (r(\widehat{B})+1 )\times M_{\widehat{B}}]=\{O_{(j,k,l)}\in \mathcal{A}_{\widehat{B}}:j<\omega,  k<r(\widehat{B})+1,l<M_{\widehat{B}}\}$. Say $\mathcal{O}_{\widehat{B}}=\{O_{(j,k,l)}\in \mathcal{A}_{\widehat{B}}:j<\omega,  k<r(\widehat{B})+1,l<M_{\widehat{B}}\}$.

	Let $\sigma$ be a  winning strategy for \2 in $BM(X)$. Now, we construct
	a  2-tactic $t$ for \2  by describing how s/he
	makes moves in an arbitrary play of $BM(X)$ (Remark: We have supposed that  $U_i\in pur(\tau^*(X))$ for any move $U_i$ of \1).

	To begin first round, \1 makes a move  $U_0\in pur(\tau^*(X))$. Then we have $\widehat{U'_0}\in pur(\mathcal{B})$ and (because of 3-Definition)  a disjoint family  $\mathcal{O}_{\widehat{U'_0}}=\{O^0_{(j,k,l)}\in \mathcal{A}_{\widehat{U'_0}}:j<\omega,  k<r(\widehat{U'_0})+1,l<M_{\widehat{U'_0}}\}$. Define $t(U_0)=\sigma(O^0_{(0,0,0)})$ and \2 plays $t(U_0)$ to complete this round.  
	Thus, originally, this play  part  $$(U_0,t(U_0)=V_0)$$  is formed in $BM(X)$  (this play part will be an arbitrary   play which is  played by \2 using $t$ and this play  will be used to prove that $t$ is a  2-tactic).  In addition, this  play  part \begin{equation}
		(O^0_{(0,0,0)}, \sigma(O^0_{(0,0,0)})=V_0)	
	\end{equation} 
	is also formed in $BM(X)$.

(To easy reading we write $\mathcal{O}_{\widehat{U'_0}}=\{O^0_{(j,k,l)}\in \mathcal{A}_{\widehat{U'_0}}:...\}$ instead of $\mathcal{O}_{\widehat{U'_0}}=\{O_{(j,k,l)}\in \mathcal{A}_{\widehat{U'_0}}:...\}$. As the reader will notice, this $0$ cannot actually be seen. But it is not a problem because  the $0$ of $O^0_{(j,k,l)}$ is not essential for this proof.  We do same thing for the rest of this proof.) 
	
	To begin round 2, \1 makes a move $U_1\subseteq t(U_0)$. \2  sees $U_0$ and $U_1$. So, \2 recognizes  $\mathcal{O}_{\widehat{U'_0}}$. Because
	$\mathcal{O}_{\widehat{U'_0}}$ consists of pairwise disjoint sets (and $U_1\subseteq t(U_0)=\sigma(O^0_{(0,0,0)} )\subseteq O^0_{(0,0,0)}$ ),  $U_1$ is a subset of exactly one of these sets, which is $O^0_{{(0,0,0)}}$. For \2, the code $(0,0,0)$ of $O^0_{(0,0,0)}$ means that it is the beginning of the play. Thus,  by using $\sigma$,  \2 is able to reconstruct the  play part $(O^0_{(0,0,0)}, \sigma(O^0_{(0,0,0)})=V_0)$. Note that from the definition of $h$, $o(\widehat{U'_0})\leq o(\widehat{U'_1})$ and $r(\widehat{U'_0})\leq r(\widehat{U'_1})$.  
	Thus, $\widehat{U'_0}\in\mathcal{C}_{\widehat{U'_1},r(\widehat{U'_0})}=\{V\in  \mathcal{B}_{r(\widehat{U'_0})}:\widehat{U'_1}\subseteq V, o(V)\leq o(\widehat{U'_1})\}$. Since we know
	$\mathcal{C}_{\widehat{U'_1},r(\widehat{U'_0})}=\{V_l:l \in m_{\widehat{U'_1},r(\widehat{U'_0})}\}$, there is an
	$l_0\in m_{\widehat{U'_1},r(\widehat{U'_0})}\leq M_{\widehat{U'_1}}$ 
	such that  $\widehat{U'_0}=V_{l_0}$. So, there exists
	$O^1_{(1,r(\widehat{U'_0}),l_0)}\in\mathcal{O}_{\widehat{U'_1}}=\{O^1_{(j,k,l)}\in \mathcal{A}_{\widehat{U'_1}}:j<\omega,  k<r(\widehat{U'_1})+1,l<M_{\widehat{U'_1}}\}$. 
	Therefore,  \2 makes this move $$t(U_0,U_1)=\sigma(O^0_{(0,0,0)}, V_0, O^1_{(1,r(\widehat{U'_0}),l_0)})=V_1.$$ 
	Thus, originally, this play part  $$(U_0,t(U_0)=V_0, U_1, t(U_0,U_1)=V_1)$$ is formed in $BM(X)$ and this round is completed. 
	Also, this  play part    
	$$(O^0_{(0,0,0)},V_0, O^1_{(1,r(\widehat{U'_0}),l_0)}, \sigma(\ldots, O^1_{(1,r(\widehat{U'_0}),l_0)})=V_1)$$
	is formed  in $BM(X)$.

	To begin any further round, where $m\geq2$,	\1 makes a move $U_m\subseteq V_{m-1}$.  \2  sees $U_m$ and $U_{m-1}$. Then,
	\2 has $\mathcal{O}_{\widehat{U'_{m-1}}}$. Since
	$\mathcal{O}_{\widehat{U'_{m-1}}}$ consists of pairwise disjoint sets and $U_m$ is a 
	subset of exactly one of them, 
	which is
	$O^{m-1}_{(m-1, r(\widehat{U'_{m-2}}), l_{m-2})}$. 
	For \2, the components $m-1$, $r(\widehat{U'_{m-2}})$  and $l_{m-2}$ of $(m-1,r(\widehat{U'_{m-2}}),l_{m-2})$ stand for the round number, the rank of 
	$\widehat{U'_{(m-2)}}$ in  
	$pur(\mathcal{B})$ and the index number of $\widehat{U'_{(m-2)}}$ in $\mathcal{C}_{\widehat{U'_{m-1}},r(\widehat{U'_{m-2}})}$, respectively. So, \2 finds out  $\widehat{U'_{m-2}}=V_{l_{m-2}}\in \mathcal{C}_{\widehat{U'_{m-1}},r(\widehat{U'_{m-2}})}$. Now, since \2 has $\widehat{U'_{m-2}}$ and  $\widehat{U'_{m-1}}$, similarly (because  $\mathcal{O}_{\widehat{U'_{m-2}}}$ is disjoint family,  
	$\widehat{U'_{m-1}}\subseteq O^{m-2}_{(m-2,r(\widehat{U'_{m-3}}),l_{m-3})}$ and so, \2  has the components  of $(m-2,r(\widehat{U'_{m-3}}),l_{m-3})$, then \2  gets $\widehat{U'_{m-3}}=V_{l_{m-3}}\in \mathcal{C}_{\widehat{U'_{m-2}},r(\widehat{U'_{m-3}})}$),
	\2 finds out  $\widehat{U'_{m-3}}$. In a similar manner,  \2 keeps going until to get the code $(0,0,0)$ and  so, \2 also finds out $\widehat{U'_{m-4}},\widehat{U'_{m-5}}, \ldots, \widehat{U'_1},\widehat{U'_0}$   and
	$O^{m-4}_{(m-4,r(\widehat{U'_{m-5}}),l_{m-5})},\ldots, O^{1}_{(1,r(\widehat{U'_{0}}),l_0)}, O^{0}_{(0,0,0)}$.
	Therefore, by using $\sigma$, \2 is able
	to reconstruct the play part 
	$$(O^0_{(0,0,0)},V_0, O^1_{(1,r(\widehat{U'_{0}}),l_0)}, \ldots, O^{m-1}_{(m-1,r(\widehat{U'_{m-2}}),l_{m-2})}, V_{m-1})$$
	where $V_i=\sigma(O^0_{(0,0,0)},\ldots, O^i_{(i,r(\widehat{U'_{i-1}}),l_{i-1})})$ 
	for all $0<i\leq m-1$ and $V_0=\sigma(O^0_{(0,0,0)})$. Note that from definition of $h$ (in 1-Definition) there exists $l_{m-1}\in m_{\widehat{U'_{m}},r(\widehat{U'_{m-1}})}\le M_{\widehat{U'_{m}}}$ 
	such that  $\widehat{U'_{m-1}}=V_{l_{m-1}}\in \mathcal{C}_{\widehat{U'_{m}},r(\widehat{U'_{m-1}})}$  
	because $\widehat{U'_{m-1}}\supseteq \widehat{U'_{m}}$. So, $O^m_{(m,r(\widehat{U'_{m-1}}),l_{m-1})}\in \mathcal{O}_{\widehat{U'_m}}=\{O^m_{(j,k,l)}\in \mathcal{A}_{\widehat{U'_m}}:j<\omega,  k<r(\widehat{U'_m})+1,l<M_{\widehat{U'_m}}\}$. So, \2 can  make this move $$t(U_{m-1},U_m)=\sigma(O^0_{(0,0,0)},V_0, O^1_{(1,r(\widehat{U'_{0}}),l_0)}, \ldots, V_{m-1}, O^{m}_{(m,r(\widehat{U'_{m-1}}),l_{m-1})})=V_m.$$ 
	
	Eventually, this play in which \2 plays with $t$  will be formed like this:  $$(U_0,t(U_0)=V_0, U_1,  t(U_0,U_1)=V_1,U_2, t(U_1,U_2)=V_2,\ldots)$$ in $BM(X)$ (to  see  $t$ is a  2-tactic for \2, we need to prove that $\emptyset\neq \bigcap_{i\in \omega}V_i$). 
	In addition, there  will be a  play   
	
	$$(O^0_{(0,0,0)},V_0, O^1_{(1,r(\widehat{U'_{0}}),l_0)}, V_1, O^{2}_{(2,r(\widehat{U'_{1}}),l_1)}, V_{2},\ldots)$$ where $V_0=\sigma(O^0_{(0,0,0)})$ and  $V_i=\sigma(O^0_{(0,0,0)},\ldots, O^{i}_{(i,r(\widehat{U'_{i-1}}),l_{i-1})})$ for all $0< i< \omega$ in $BM(X)$. Since $\sigma$ is a winning strategy for \2, from the second  play above, $\emptyset\neq \bigcap_{i\in \omega}V_i$.

\end{proof}

We use in the proof above a technique which is one of the main techniques used when one strategy is derived from another in topological games.  The main idea is to produce two entwined plays, one to be played with known strategy and the other with the strategy to be defined. 

\begin{remark}\label{rem} ($\star$) and ($\star\star$)-conditions can be different from each other.  Because, from Remark \ref{bitmedi} and Example \ref{gitti}, we  know that a Noetherian table of a Noetherian family $\mathcal{A}$ can be different from a Noetherian table of a subfamily of $\mathcal{A}$, even though  from Theorem \ref{yeter} we know something about them. Note that $[\mathcal{B}]$ is a Noetherian table of $\mathcal{B}$ and $[pur(\mathcal{B})]$ is a Noetherian table of this subfamily  $pur(\mathcal{B})\subseteq \mathcal{B}$. Also, there is more to these conditions than a subfamily. The same is valid for $\dagger$ and $\dagger\dagger$ in Corollary \ref{Teorem winning ten 2-tactige ana sonuc}.
\begin{flushright}
	\kare
\end{flushright}	
\end{remark}

The following corollary states that any Galvin  $\pi$-base can be modified to satisfy the hypothesis of Theorem \ref{Teorem winning ten 2-tactige ana sonuc 2}. So, Theorem \ref{Teorem winning ten 2-tactige ana sonuc 2} includes  the Galvin's theorem (see the introduction of this article for the Galvin's theorem).  Most likely, Theorem \ref{Teorem winning ten 2-tactige ana sonuc 2} says more than the Galvin's theorem. But to show this we need a proper example. We left it as an open question (see Question \ref{Galvin}).

\begin{corollary}\label{vin}
	If any $\pi$-base $\mathcal{B}$ for a space $X$ has the following (*)-property (the hypothesis of Galvin's theorem (in our terminology a Galvin $\pi$-base)), then $X$ has a $\pi$-base $\mathcal{B}'\subseteq \mathcal{B}$ such that $\mathcal{B}'$ has ($\star$) and ($\star\star$)-conditions of Theorem \ref{Teorem winning ten 2-tactige ana sonuc 2}.
	
	(*) for any $B\in \mathcal{B}$  there exists a cellular family 
	$\mathcal{A}_B\subseteq \tau(B)$  such that $|\mathcal{A}_B|\geq  |\{V \in\mathcal{B}: B \subseteq V \}|$.
\end{corollary}

\begin{proof}

	From Corollary \ref{k}, there exist a Noetherian $\pi$-base $\mathcal{B}'\subseteq \mathcal{B}$. Fix any  Noetherian table  $[\mathcal{B}']$  of
	 $\mathcal{B'}$. 
	(Note that from Theorem \ref{yeter}  if the subfamily $pur(\mathcal{B'})$  of $\mathcal{B'}$   is not emptyset, then $pur(\mathcal{B'})$ has  a  Noetherian table $[pur(\mathcal{B'})]$, so fix any  Noetherian table  $[pur(\mathcal{B'})]$  of
	$pur(\mathcal{B'})$)

We will see that $\mathcal{B'}$ has the conditions of Theorem \ref{Teorem winning ten 2-tactige ana sonuc 2}.	If  $pur(\mathcal{B'})=\emptyset$   then the proof is completed. If not, fix any $B\in pur(\mathcal{B'})$. From (*)-property there exists a disjoint family $\mathcal{A}_{B}\subseteq \tau^*(B)$ such that $|\mathcal{A}_{B}|\geq  |\{V \in\mathcal{B}: B \subseteq V \}|$. (To see $\mathcal{B'}$ has ($\star$)-condition, say $B=U^{o'(B)}_{r'(B)}\in [pur(\mathcal{B'})]$ or to see $\mathcal{B'}$ has ($\star\star$)-condition, say $B=U^{o(B)}_{r(B)}\in [\mathcal{B'}]$. We will use $r(B)$ 
and $o(B)$, so, the proof will be valid for both conditions (($\star$) or in the remark ($\star\star$)) of Theorem \ref{Teorem winning ten 2-tactige ana sonuc 2}.)  Firstly, we will  see that $|\mathcal{A}_{B}|\geq |r(B)|$.  If $r(B)=0$, then it is clear.  If $r(B)>0$, from Lemma \ref{L: N}, for any $0\leq i<r(B)$, we can
 choose a $B_i\in \{V \in\mathcal{B}: B \subseteq V \}$ such that $r(B_i)=i$  and so, the function $f:r(B)\rightarrow\{V \in\mathcal{B}: B \subseteq V \}$,  $f(i)=B_i$,   is one to one. Thus, $|\{V \in\mathcal{B}: B \subseteq V \}|\geq |r(B)|$. So, $|\mathcal{A}_{B}|\geq |r(B)|$. Now, for any $\theta\leq r(B)$, 
	$\{V\in  \mathcal{B}_\theta:B\subseteq V, o(V)\leq o(B)\}\subseteq \{V \in\mathcal{B}: B \subseteq V \}$. Thus, $M_B=\sup\{|\{V\in  \mathcal{B}_\theta:B\subseteq V, o(V)\leq o(B)\}|:\theta\leq r(B)\}\leq |\{V \in\mathcal{B}: B \subseteq V \}|$. Therefore, $|\mathcal{A}_{B}|\geq M_B$. Hence, $|\mathcal{A}_{B}|\geq \max\{|r(B)|,M_B\}$.
			
\end{proof}

We give two different proofs  for Corollary \ref{Teorem winning ten 2-tactige ana sonuc}.  Our purpose in making the second proof  is to show that Noetherian bases and tables provide an another proof way which is  somewhat  different from the way used in the proof of Theorem \ref{Teorem winning ten 2-tactige ana sonuc 2}.

\begin{corollary}\label{Teorem winning ten 2-tactige ana sonuc} 
Let $X$ be a space and  $\mathcal{B}$ a $\pi$-base of $X$ having the following ($\dagger$)-condition. If \2 has a winning strategy in $BM(X)$, then \2 has a  2-tactic in $BM(X)$.

($\dagger$)  $pur(\mathcal{B})=\emptyset$  or  $pur(\mathcal{B})$ is   Noetherian and has a 
Noetherian table $[pur(\mathcal{B})]$ such that
for any $B\in [pur(\mathcal{B})]$
there exists a disjoint family $\mathcal{A}_{B}\subseteq \tau^*(B)$ such that 
$|\mathcal{A}_{B}|\geq \max\{|r(B)|,|o(B)|\}$.

\end{corollary}

\textit{Remark.} This theorem is also true if ($\dagger$)-condition is replaced by the following ($\dagger\dagger$)-condition. Note that where $B\in pur(\mathcal{B})$, $r(B)$ and $o(B)$ depend on the Noetherian table $[pur(\mathcal{B})]$ in ($\dagger$)-condition while $r(B)$ and $o(B)$ depend on the Noetherian table $[\mathcal{B}]$ in ($\dagger\dagger$)-condition.

($\dagger\dagger$) $pur(\mathcal{B})=\emptyset$ or where $pur(\mathcal{B})\neq\emptyset$, $\mathcal{B}$ is  Noetherian and has a 
Noetherian table $[\mathcal{B}]$ such that
for  any  $B\in pur(\mathcal{B})$, suppose  $B=U^{o(B)}_{r(B)}\in [\mathcal{B}]$ then,
there exists a disjoint family $\mathcal{A}_{B}\subseteq \tau^*(B)$ such that 
$|\mathcal{A}_{B}|\geq \max\{|r(B)|,|o(B)|\}$.

\begin{proof}

	(\textbf{Proof 1}) We will see that $\mathcal{B}$ satisfies conditions of Theorem \ref{Teorem winning ten 2-tactige ana sonuc 2}. If $pur(\mathcal{B})=\emptyset$, the proof is completed. If not,
take any	$B\in pur(\mathcal{B})$ then from the hypothesis and definition of $pur(\mathcal{B})$,  there exists a disjoint family 
	$\mathcal{A}_B\subseteq \tau^*(B)$  such that $|\mathcal{A}_B|\geq  \max\{|r(B)|,|o(B)|, \omega\}$. Say $\theta\leq r(B)$. Because  
the function $f:\{V \in\mathcal{B}_\theta: B \subseteq V, o(V)\leq o(B) \}\rightarrow o(B)+1$,  $f(V)=o(V)$,   is one to one, $|\{V \in\mathcal{B}_\theta: B \subseteq V, o(V)\leq o(B)  \}|\leq |o(B)+1|$. Thus, $M_B=\sup\{|\{V\in \mathcal{B}_\theta:B\subseteq V, o(V)\leq o(B) \}|:\theta\leq r(B)\}\leq |o(B)+1|$. So, $|\mathcal{A}_{B}|\geq M_B$ because $|\mathcal{A}_{B}|\geq \max\{|o(B)|,\omega\}$. Hence $|\mathcal{A}_{B}|\geq \max\{|r(B)|,M_B\}$. Therefore, from Theorem \ref{Teorem winning ten 2-tactige ana sonuc 2}, the proof is completed.

		(\textbf{Proof 2}) From Corollary \ref{c:makale}, the proof is completed where $pur(\mathcal{B})=\emptyset$.
		Suppose $pur(\mathcal{B})\neq \emptyset$. Let  $[pur(\mathcal{B})]$ be an  $\alpha\times\beta$ Noetherian table and say Noetherian union of $pur(\mathcal{B})$ is	$\bigcup_{i<\alpha}\mathcal{B}_i$.
Say any $i$-th row of  $[pur(\mathcal{B})]$ is $(B^0_i,B^1_i, B^2_i, \ldots)$.

Take any $B^j_i\in [pur(\mathcal{B})]$. From the hypothesis and definition of $pur(\mathcal{B})$,  there exists a disjoint family 
$\mathcal{A}_{B^j_i}\subseteq \tau^*(B)$  such that $|\mathcal{A}_{B^j_i}|\geq  \max\{|r(B)|,|o(B)|, \omega\}$. So, by fixing a one to one function $f$ from the set  $\bigcup_{n\in \omega-\{0\}}((i+1) \times (j+1))^n$ to disjoint family 
$\mathcal{A}_{B^j_i}$  where $((i+1) \times (j+1))^n=\{((a_1,b_1), (a_2,b_2),\ldots,(a_n,b_n)):\forall 1\leq k\leq n(a_k<i+1,  b_k<j+1)\}$, we can enumerate the subset  $f[\bigcup_{n\in \omega-\{0\}}((i+1) \times (j+1))^n]$ of $\mathcal{A}_{B^j_i}$ such 
 that $f[\bigcup_{n\in \omega-\{0\}}((i+1) \times (j+1))^n]=\{O_{((a_1,b_1), (a_2,b_2),\ldots,(a_n,b_n))}\in \mathcal{A}_{B^j_i}:n \in \omega- \{0\} ,\forall 1\leq k\leq n (a_k<i+1,  b_k<j+1) \}$. Say $\mathcal{O}_{B^j_i}=\{O_{((a_1,b_1), (a_2,b_2),\ldots,(a_n,b_n))}\in \mathcal{A}_{B^j_i}:n \in \omega- \{0\} ,\forall 1\leq k\leq n (a_k<i+1,  b_k<j+1) \}$. Note that $i$, $j$ are ordinals and
 $i + 1$, $j + 1$ are ordinal successors of $i$ and $j$.

 Now, take any $B^j_i\in [pur(\mathcal{B})]$ say $h(B^j_i)=\min \{o(B)\in \beta: B\in pur(\mathcal{B}), B\subseteq B^j_i\}$  and $s(B^j_i)= \min \{s\in \alpha: B^{h{(B^j_i)}}_s\in [pur(\mathcal{B})], B^{h{(B^j_i)}}_s\subseteq B^j_i\}$. 
 Define  $\widehat{B^j_{i}}=B^{h(B^j_i)}_{s{(B^j_i)}}$.

Now, put a well order on $\mathcal{B}$ and say $U'=\min \{B\in \mathcal{B}: B\subseteq U\}$
where $U\in \tau^*(X)$. From Remark \ref{r:g},	without loss of generality, we may assume that in each  round of any play if $U_i$ is a move  of \1,  then $U_i\in pur(\tau^{*}(X))$. 
Note that $U'_i\in pur(\mathcal{B})$ if $U_i$ is a move of \1 because $U_i \in pur(\tau^{*}(X))$

From now on,  where $i\leq n$, throughout this proof for abbreviation  $a_i\asymp b_n$ and $r(\widehat{U'_i})\asymp o(\widehat{U'_{n}})$  stand for  $(a_i,b_i), (a_{i+1},b_{i+1}),(a_{i+2},b_{i+2}), \ldots,(a_n,b_n)$ and $(r(\widehat{U'_{i}}),o(\widehat{U'_{i}})),(r(\widehat{U'_{i+1}}),o(\widehat{U'_{i+1}})), \ldots,(r(\widehat{U'_{n}}),o(\widehat{U'_{n}}))$, respectively.

Let $\sigma$ be a  winning strategy for \2 in $BM(X)$. Now, we construct
a   2-tactic $t$ for \2  by describing how s/he
makes moves in an arbitrary play of $BM(X)$.

To begin first round, \1 makes a move  $U_0\in pur(\tau^{*}(X))$. Then  we have $\widehat{U'_0}\in pur(\mathcal{B})$ (note that $U'_0=B^j_i\in [pur(\mathcal{B})]$ for some ordinals $i,j$ and $\widehat{U'_0}=\widehat{B^j_i}=B^{h(B^j_i)}_{s(B^j_i)}\in [pur(\mathcal{B})]$)
 and a disjoint family  $\mathcal{O}_{\widehat{U'_0}}=\{O^0_{(a_1\asymp b_n)}:n \in \omega- \{0\},\forall 1\leq k\leq n, (a_k<r(\widehat{U'_0})+1,  b_k<o(\widehat{U'_0})+1) \}\subseteq \mathcal{A}_{\widehat{U'_0}}$. Define $t(U_0)=\sigma(O^0_{(r(\widehat{U'_0}),o(\widehat{U'_0}))})$ and \2 plays $t(U_0)$ to complete this round.  
Thus, originally, this play  part  $$(U_0,t(U_0)=V_0)$$ is formed in $BM(X)$  (this play part will be an arbitrary   play which is  played by \2 using $t$ and this play  will be used to prove that $t$ is a  2-tactic).  In addition, this  play  part \begin{equation}\label{1}
 (O^0_{(r(\widehat{U'_0}),o(\widehat{U'_0}))}, \sigma(O^0_{(r(\widehat{U'_0}),o(\widehat{U'_0}))})=V_0)	
\end{equation} 
is also formed in $BM(X)$.

To begin round 2, \1 makes a move $U_1\subseteq t(U_0)$. \2  sees $U_0$ and $U_1$. So, \2 recognizes  $\mathcal{O}_{\widehat{U'_0}}$. Because
$\mathcal{O}_{\widehat{U'_0}}$ consists of pairwise disjoint sets (and $U_1\subseteq t(U_0)=\sigma(O^0_{(r(\widehat{U'_0}),o(\widehat{U'_0}))} )\subseteq O^0_{(r(\widehat{U'_0}),o(\widehat{U'_0}))}$ ),  $U_1$ is a subset of exactly one of these sets, which is $O^0_{(r(\widehat{U'_0}),o(\widehat{U'_0}))}$. Note that   $(r(\widehat{U'_0}),o(\widehat{U'_0}))$ is coordinates of $\widehat{U'_0}$ in $[pur(\mathcal{B})]$, i.e., $\widehat{U'_0}=B^{o(\widehat{U'_0})}_{r(\widehat{U'_0})}$. Since no more coordinates, \2 only has  $\widehat{U'_0}$ and so,  \2 is able to reconstruct play part (1) by using $\sigma$. Note that $o(\widehat{U'_0})\leq o(\widehat{U'_1})$ and $r(\widehat{U'_0})\leq r(\widehat{U'_1})$ from the definitions of $h$ and $s$. Thus, $O^1_{(r(\widehat{U'_0})\asymp o(\widehat{U'_1}))}\in\mathcal{O}_{\widehat{U'_1}}=\{O^1_{(a_1\asymp b_n)}\in \mathcal{A}_{\widehat{U'_1}}:n \in \omega- \{0\},\forall 1\leq k\leq n, (a_k<r(\widehat{U'_1})+1,  b_k<o(\widehat{U'_1})+1)\}$ 
 and so,  \2 makes this move $$t(U_0,U_1)=\sigma(O^0_{(r(\widehat{U'_0}),(o(\widehat{U'_0}))}, V_0, O^1_{(r(\widehat{U'_0})\asymp o(\widehat{U'_1}))})=V_1.$$ 
 Thus, originally, this play part  $$(U_0,t(U_0)=V_0, U_1, t(U_0,U_1)=V_1)$$ is formed in $BM(X)$ and this round is completed. Also, this  play part    
$$(O^0_{(r(\widehat{U'_0}),o(\widehat{U'_0}))},V_0, O^1_{(r(\widehat{U'_0})\asymp o(\widehat{U'_1}))}, \sigma(\ldots, O^1_{(r(\widehat{U'_0})\asymp o(\widehat{U'_1}))})=V_1)$$
is formed  in $BM(X)$.

All further rounds are played in a similar manner to round 2. 
\1 makes a move $U_m\subseteq V_{m-1}$.  \2  sees $U_m$ and $U_{m-1}$. Then,
\2 has $\mathcal{O}_{\widehat{U'_{m-1}}}$. Since
$\mathcal{O}_{\widehat{U'_{m-1}}}$ consists of pairwise disjoint sets, $U_m$ is a 
subset of exactly one of them, 
which is
$O^{m-1}_{(r(\widehat{U'_0})\asymp o(\widehat{U'_{m-1}}))}$. 
Note that  $r(\widehat{U'_0})\asymp o(\widehat{U'_{m-1}})=(r(\widehat{U'_0}),o(\widehat{U'_0})),\ldots,(r(\widehat{U'_{m-1}}),o(\widehat{U'_{m-1}}))$ and each $(r(\widehat{U'_k}),o(\widehat{U'_k}))$ is coordinates of $\widehat{U'_k}$ in $[pur(\mathcal{B})]$, i.e., $\widehat{U'_k}=B^{o(\widehat{U'_k})}_{r(\widehat{U'_k})}$. Thus, \2 notices $\widehat{U'_0},\widehat{U'_1},\ldots,\widehat{U'_{m-1}}$ and so, by using disjonity of $\mathcal{O}_{\widehat{U'_k}}$ families,  also notices $O^0_{(r(\widehat{U'_0}),o(\widehat{U'_0}))}, O^1_{(r(\widehat{U'_0})\asymp o(\widehat{U'_1}))}, \ldots, O^{m-1}_{(r(\widehat{U'_0})\asymp o(\widehat{U'_{m-1}}))}$. Therefore \2 is able
 to reconstruct the play part 
$$(O^0_{(r(\widehat{U'_0}),o(\widehat{U'_0}))},V_0, O^1_{(r(\widehat{U'_0})\asymp o(\widehat{U'_1}))}, \ldots, O^{m-1}_{(r(\widehat{U'_0})\asymp o(\widehat{U'_{m-1}}))}, V_{m-1})$$
where $V_i=\sigma(O^0_{(r(\widehat{U'_0}),o(\widehat{U'_0}))},\ldots, O^i_{(r(\widehat{U'_i})\asymp o(\widehat{U'_i}))})$ 
for all $0\leq i\leq m-1$. Note  that $O^m_{(r(\widehat{U'_0})\asymp o(\widehat{U'_m}))}\in \mathcal{O}_{\widehat{U'_m}}=\{O^m_{((a_1\asymp b_m))}:n \in \omega- \{0\},\forall 1\leq k\leq n, (a_k<r(\widehat{U'_m})+1,  b_k<o(\widehat{U'_m})+1)\}$. So, \2 can  make this move $V_m=t(U_{m-1},U_m)$ where $V_m=\sigma(O^0_{(r(\widehat{U'_0}),o(\widehat{U'_0}))},V_0,\ldots,$ $ O^{m-1}_{(r(\widehat{U'_0})\asymp o(\widehat{U'_{m-1}}))},V_{m-1}, O^m_{(r(\widehat{U'_0})\asymp o(\widehat{U'_m}))}))$. 

Eventually, this play in which \2 plays with $t$  will be formed like this:  $$(U_0,t(U_0)=V_0, U_1,  t(U_0,U_1)=V_1,U_2, t(U_1,U_2)=V_2,\ldots)$$ in $BM(X)$ (to  see  $t$ is a  2-tactic for \2, we need to prove that $\emptyset\neq \bigcap_{i\in \omega}V_i$). 
In addition, there  exists a  play   

$$(O^0_{(r(\widehat{U'_0}),o(\widehat{U'_0}))}, V_0, O^1_{(r(\widehat{U'_0})\asymp o(\widehat{U'_1}))},V_1, O^2_{(r(\widehat{U'_0})\asymp o(\widehat{U'_2}))}, \ldots)$$ where $V_i=\sigma(O^0_{(r(\widehat{U'_0}),o(\widehat{U'_0}))},\ldots, O^i_{(r(\widehat{U'_0})\asymp o(\widehat{U'_i}))})$ for all $0\leq i< \omega$ in $BM(X)$. Since $\sigma$ is a winning strategy for \2, from the second  play above, $\emptyset\neq \bigcap_{i\in \omega}V_i$. 
\end{proof}

It is the corollary above that prompted the writing of this article. Before almost all other results, we sensed Corollary \ref{Teorem winning ten 2-tactige ana sonuc} and we proved it. Then Theorem \ref{Teorem winning ten 2-tactige ana sonuc 2} was revealed. We thought of removing Corollary \ref{Teorem winning ten 2-tactige ana sonuc} from this paper. But we decided that it would be better to keep it  since it has a different use of the Noetherian $\pi$-bases in its proof and one of the purposes of this article is to show that Noetherian families are  useful for topological games. In addition, Corollary \ref{Teorem winning ten 2-tactige ana sonuc} may   be used for revealing the difference between  Theorem \ref{Teorem winning ten 2-tactige ana sonuc 2} and the Galvin's theorem (see Section \ref{Ack sorular}). 

\begin{remark}\label{R:1}
Let us be given any $\pi$-base $\mathcal{B}$ of a space $X$. If $\mathcal{B}$ satisfies the conditions of Corollary \ref{Teorem winning ten 2-tactige ana sonuc}, then $\mathcal{B}$ satisfies the conditions of Theorem \ref{Teorem winning ten 2-tactige ana sonuc 2}. Because the proof-1 of Corollary \ref{Teorem winning ten 2-tactige ana sonuc} implies that. \begin{flushright}
	\kare
\end{flushright}
\end{remark}
 
 The following lemma might be familiar to some.
 
 \begin{lemma} \label{key1}
 	Let $\mathcal{B}$ be a $\pi$-base of a space  $X$.  Then, there exists a  $\subseteq$-maximal disjoint subfamily $\mathcal{A}$ of  $\mathcal{B}$. (Note that then $X= \overline{ \cup \mathcal{A}}$.)

 \end{lemma}
 
  \begin{proof}
    Let $|\mathcal{B}|=\kappa$ and $\mathcal{B}=\{O_i: i\in \kappa\}$. 

 A Choosing  Procedure: For $i=0$, define $m_0=0$ and $\mathcal{A}_0=\{O_0\}=\{O_{m_0}\}$.  For $i=1$, if $\overline{ \cup \mathcal{A}_0}=X$, then define $\mathcal{A}=\bigcup_{j<1}\mathcal{A}_j$ and finish the procedure. If not, set $m_1=\min \{t\in \kappa: O_t\subseteq X- \overline{\cup \bigcup_{j<1}\mathcal{A}_j}\}$ and define $\mathcal{A}_1=\{O_{m_1}\}\cup(\bigcup_{j<1}\mathcal{A}_j)$. Now, for $i \in \kappa$ where $1\leq j<i< \kappa$  for all $j$,   suppose $\mathcal{A}_j$  has been  defined and the procedure was not finished.   If $\overline{\cup\bigcup_{j<i}\mathcal{A}_j}=X$, then define $\mathcal{A}=\bigcup_{j<i}\mathcal{A}_j$ and finish   the procedure. If not, set  $m_i=\min \{t\in \kappa: O_t\subseteq X- \overline{ \cup \bigcup_{j<i}\mathcal{A}_j}\}$ and define $\mathcal{A}_i=\{O_{m_i}\}\cup(\bigcup_{j<i}\mathcal{A}_j)$.
 
  Thus, either this procedure will finish for some $i \in \kappa$, in which case we obtain $\mathcal{A} = \bigcup_{j < i} \mathcal{A}_j$, or it will continue for all $i \in \kappa$. If it continues for all $i \in \kappa$, then define $\mathcal{A} = \bigcup_{i < \kappa} \mathcal{A}_i$. In any case, we obtain the family $\mathcal{A}\subseteq \mathcal{B}$.
 
 In the procedure above, by using transfinite induction, it can be seen that \(i \leq m_i\) and \(O_i \cap (\cup \mathcal{A}_i) \neq \emptyset\) for all \(i \in \kappa\) where this procedure has not finished for \(i\). Thus, it is straightforward to see that  $\mathcal{A}$   is  disjoint and $\subseteq$-maximal subfamily of  $\mathcal{B}$ and so,  $X= \overline{ \cup \mathcal{A}}$.

   \end{proof}

\begin{theorem}\label{ne}
For any space $X$ the followings are equivalent.

(i) \( X \) has a Galvin \(\pi\)-base, i.e., \( X \) has a \(\pi\)-base that satisfies the (*)-property.

 (ii) The collection of all $O\in \tau^*(X)$ with the following "property" forms a $\pi$-base for $X$.
 
"There exists a $\pi$-base \(\mathcal{B}_O\) for the subspace \( O \) such that for all \( W \in \tau^*(O) \), there exists a  disjoint family $\mathcal{A}\subseteq \tau^*(O)$ satisfying $|\{ B \in \mathcal{B}_O : W \subseteq B \}|\leq |\mathcal{A}|.$"

(iii) Any $U\in \tau^*(X)$ has the "property" in (ii).

\end{theorem}

\begin{proof}
	$(i\Rightarrow ii)$ Let $\mathcal{B}$ be a Galvin base for $X$. We will see that any  $O \in\mathcal{B}$ satisfies the conditions of (ii). For any $O \in \mathcal{B}$, define $\mathcal{B}_O=\{B\in \mathcal{B}: B\subseteq O\}$. Then,    $\mathcal{B}_O$ is a $\pi$-base for subspace $O$. For any \( W \in \tau^*(O) \), there is a $W'\in  \mathcal{B}$ such that $W'\subseteq W$. Because $ \mathcal{B}$ is a Galvin base, there exists a disjoint family \(\mathcal{A}_{W'} \subseteq \tau^*(W') \)   satisfying \( |\{ B \in \mathcal{B} : W' \subseteq B \}| \leq |\mathcal{A}_{W'}| \). Let $\mathcal{A}_{W'}=\mathcal{A}$. 
   	Because $\{ B \in \mathcal{B}_O : W \subseteq B \}\subseteq \{ B \in \mathcal{B} : W' \subseteq B \}$,  $|\{ B \in \mathcal{B}_O : W \subseteq B \}|\leq |\mathcal{A}|$.

	$(ii\Rightarrow i)$ 	
For any $W\in \tau^*(X)$, set $\min\{S(U): U\in \tau^*(W)\}=\gamma_W$. Then, choose  a $U_W\in \tau^*(W)$ such that $S(U_W)=\gamma_W$. Thus, for any nonempty open subset $B\subseteq U_W$,   $S(B)=S(U_W)$.
 	 Now,
	 define $\{U_{W}: W \in \tau^*(X) \}$. From Lemma \ref{key1}, there exists a disjoint,  $\subseteq$-maximal  subfamily $\mathcal{U}$ of $\{U_{W}: W \in \tau^*(X) \}$ satisfying $\overline{\cup \mathcal{U}}=X.$

Let $\mathcal{B}$ be the $\pi$-base for $X$ such that  each element of $\mathcal{B}$  has the "property" in $(ii)$.  For any $U\in \mathcal{U}$, define $\mathcal{B}(U)=\{O\in\mathcal{B}: O\subseteq U \}$. Because 	$\mathcal{B}(U)$ is a $\pi$-base of $U$, from Lemma \ref{key1}, there exists a disjoint subfamily $\mathcal{\mathcal{D}_U}$ of $\mathcal{B}(U)$ satisfying $U\subseteq \overline{\cup \mathcal{D}_U}.$ Note
  that $\mathcal{D}_U\subseteq \mathcal{B}(U)\subseteq \mathcal{B}.$ So, from the hypothesis,  for any \( O \in\mathcal{D}_U \), there exists a \(\pi\)-base \(\mathcal{B}_O\) for  \( O \) satisfying  the conditions of the hypothesis. 
 Now, set $\mathcal{F}=\bigcup_{U\in \mathcal{U}} \mathcal{D}_U$ and define $$ \mathcal{B'}=\bigcup_{O\in \mathcal{F}} \mathcal{B}_O.$$
	
To see 	$\mathcal{B'}$ is a $\pi$-base , take any $W\in \tau^*(X)$. Because  $\overline{\cup \mathcal{U}}=X$, there is a $U\in  \mathcal{U}$  such that $U\cap W\neq \emptyset.$ Because $U\subseteq \overline{\cup \mathcal{D}_U}$, there is an $O\in \mathcal{D}_U\subseteq \mathcal{F}$ such that $O\cap W\neq \emptyset.$ Because $\mathcal{B}_O$ is a $\pi$-base of $O$, there is a $B\in \mathcal{B}_O\subseteq \mathcal{B'}$ such that $B\subseteq O\cap W$. Thus,  there is a $B\in\mathcal{B'}$ such that $B\subseteq W$.

To see $\mathcal{B'}$ is a Galvin $\pi$-base for $X$, take $B^* \in \mathcal{B'}$. Thus, there exist   $U^*\in  \mathcal{U}$ and $O^*\in \mathcal{D}_U$ such that $B^*\subseteq O^* \subseteq U^*$. Because   $\mathcal{U}$ and $\mathcal{D}_U$ disjoint, $O^*$ and $U^*$ are unique. Because $B^*\in \mathcal{B}_{O^*}$, there exists a disjoint family \(\mathcal{A}_{O^*} \subseteq \tau^*({O^*}) \)   satisfying $$ |\{ K \in \mathcal{B}_{O^*} : B^* \subseteq K \}| \leq |\mathcal{A}_{O^*}|.$$
 $S(B^*)=S(O^*)$  because  $U^*\in \mathcal{U}$ and $B^*, O^*\in \tau^*(U^*)$. Thus, there exists a disjoint $\mathcal{A}_{B^*}\subseteq \tau^*(B^*)$ such that $|\mathcal{A}_{B^*}|=|\mathcal{A}_{O^*} |$.  Therefore,  $$|\{ K \in \mathcal{B}_{O^*} : B^* \subseteq K \}| \leq |\mathcal{A}_{B^*} |.$$

	Take any $M\in \{ K \in \mathcal{B'} : B^* \subseteq K \}$. Thus, $B^* \subseteq M$ and there exist  $O\in \mathcal{F}$ and $ U\in  \mathcal{U}$ such that $M\in \mathcal{B}_O$ and $O\subseteq U$. Therefore $B^*\subseteq M\subseteq  O\subseteq U$, $ U\in  \mathcal{U}$ and $O\in \mathcal{F}$.  Because of the uniqueness of $U^*$ and $O^*$, we have $U^*=U$ and $O^*=O$. Therefore $M\in \mathcal{B}_{O^*}$. Thus, $M\in \{ K \in \mathcal{B}_{O^*}  : B^* \subseteq K \}$. So, $\{ K \in \mathcal{B}_{O^*}  : B^* \subseteq K \}=\{ K \in \mathcal{B'}  : B^* \subseteq K \}$. Therefore, $$|\{ K \in \mathcal{B'} : B^* \subseteq K \}| \leq |\mathcal{A}_{B^*} |.$$
	Hence, for any $B^* \in \mathcal{B'}$, there exists a disjoint $\mathcal{A}_{B^*}\subseteq \tau^*(B^*)$  such that $|\{ K \in \mathcal{B'} : B^* \subseteq K \}| \leq |\mathcal{A}_{B^*} |$. Thus,    $\mathcal{B'}$ is a Galvin $\pi$-base for $X$.

$(iii\Rightarrow ii)$ is obvious, because $\tau^*(X)$ is a $\pi$-base.

$(ii\Rightarrow iii)$.  Take any $U\in \tau^*(X)$. Let $\mathcal{B}$ be the $\pi$-base for $X$ such that  each $O\in \mathcal{B}$  has the "property" in $(ii)$. Let $\mathcal{B}(U)=\{B\in \mathcal{B}: B\subseteq U\}$. Thus, $\mathcal{B}(U)$ is a $\pi$-base for $U$ such that each element of   $\mathcal{B}(U)$ has the  "property" in $(ii)$. Therefore, the subspace $U$ itself satisfies $(ii)$. Because we showed that $(ii\Leftrightarrow i)$, $U$ has a Galvin $\pi$-base $\mathcal{B}_U$. Thus, for any $W\in \tau^*(U)$, there exists a disjoint family $\mathcal{A}\subseteq \tau^*(W)\subseteq \tau^*(U)$ such that $|\{B\in \mathcal{B}_U:W\subseteq B\}|\le |\mathcal{A}|$. Therefore, $U$  has the "property" in (ii).

\end{proof}

It was shown in Corollary \ref{vin} that a space with a Galvin base, i.e., a base with the (*)-property, has a base satisfying the ($\star$) and ($\star\star$)-conditions of Theorem \ref{Teorem winning ten 2-tactige ana sonuc 2}. Now, let's consider the converse: does a space with a base satisfying the ($\star$) or ($\star\star$)-conditions of Theorem \ref{Teorem winning ten 2-tactige ana sonuc 2} necessarily have a base with the (*)-property? In the following example, although it may not be entirely fair, the answer to this question is negative. To make the question more meaningful, we need to modify the (*)-property.
 
 The following example is given in \cite{brian2021telgarsky} to show that there exists a  $T_1$-space that does not have the (*)-property. Even though its proof is straightforward, we provide it here for the convenience of the reader.
 
\begin{example}\label{ds}
 Let $X$ be a set and $|X|=\omega_1$.  Let us consider the finite complement topology on $X$. Then,
 
 (i) $X$ is a $T_1$ space.
 
 (ii) \2 has a winning strategy (even 1-tactic) in $BM(X)$. Actually, any countable intersection of open sets 
 is not empty set. 
 
 (iii) For any $\pi$-base  $\mathcal{B}$ for the space  $X$,  $\mathcal{B}$  satisfies ($\star$), ($\star\star$), ($\dagger$) and  ($\dagger\dagger$)-conditions but $\mathcal{B}$ does not satisfy (*)-property.
 
 (iv) $\pi w(X)= \omega_1$.
\end{example}

\begin{proof}
(i) $X$ is a $T_1$ space,  because it has finite complement topology. 

(ii) Let $(U_0,V_0, U_1,V_1\ldots)$ be an any play in $BM(X).$ Since each of the sets $U_i$ and $V_i$ excludes only finitely many elements of $X$, it follows that $|X-{\bigcap_{i\in \omega}U_i}|\leq \omega$. But $|X|=\omega_1$. Thus, $\bigcap_{i\in \omega}U_i\neq \emptyset$.

(iii) Let $\mathcal{B}$  be a $\pi$-base for the space  $X$.

  $pur(\mathcal{B})=\emptyset$ because there is no disjoint open subsets of $X$. Because 	 $pur(\mathcal{B})=\emptyset$, then $\mathcal{B}$ satisfies ($\star$), ($\star\star$), ($\dagger$) and  ($\dagger\dagger$)-conditions.  
  
  Now, take $W_0\in \mathcal{B}$ and take any $x_0\in W_0$. Because  $\mathcal{B}$ is  a $\pi$-base and $W_0-\{x_0\}$ open subset of $X$, there exists a   $W_1 \in \mathcal{B}$ such that $W_1\subseteq W_0-\{x_0\}$. Thus
    $W_0, W_1\in \{B\in \mathcal{B}: W_1\subseteq B\}$, and so, $| \{B\in \mathcal{B}: W_1  \subseteq B\}|\geq 2$. For any  disjoint $\mathcal{A}\subseteq \tau^*(W_1)$, $|\mathcal{A}|\leq 1$ because there is no disjoint open subsets of $X$. Therefore, there exists a   $W_1 \in \mathcal{B}$ such that $| \{B\in \mathcal{B}: W_1  \subseteq B\}|>|\mathcal{A}|$ for any  disjoint $\mathcal{A}\subseteq \tau^*(W_1)$. Thus, $\mathcal{B}$ does not have the (*)-property.

 (iv) For any $\pi$-base $\mathcal{B}$ for $X$. Assume $|\mathcal{B}|\le \omega$. Then, $|\bigcup_{B\in \mathcal{B}}(X\setminus B)|\le \omega$ because $|\mathcal{B}|\le \omega$ and $(X\setminus B)$ is finite. Thus, we can take a $x\in X-\bigcup_{B\in \mathcal{B}}(X\setminus B)$. Then, there is no $B\in \mathcal{B}$ satisfying $B\subseteq X-\{x\}$ and this is a contradiction. Thus,     $|\mathcal{B}|> \omega$. Because $\mathcal{B}\subseteq \{U\subseteq X: X-U \text{ is finite}\}$ and $|\{U\subseteq X: X-U \text{ is finite}\}|=|X|=\omega_1$. Therefore, $|\mathcal{B}|= \omega_1$. Hence, $\pi w(X)= \omega_1$.
\end{proof}

With its current formulation, the (*)-property makes the problem of finding a space with a basis that satisfies both the ($\star$) and ($\star \star$)conditions but does not satisfy the 
(*)-property easily solvable, as demonstrated in the example above. Furthermore, the example shows that even  $T_1$ spaces that do not meet the (*)-property can be found. Therefore, since it becomes unsuitable to use the 
 (*)-property in the results below, we modify the property as follows.

($\bullet$) $pur(\mathcal{B})=\emptyset$ or for any $W\in pur(\mathcal{B})$ there exists a disjoint family 
$\mathcal{A}_W\subseteq \tau^*(W)$  such that $|\mathcal{A}_W|\geq  |\{B \in\mathcal{B}: W \subseteq B \}|$.

The following easy example answers some questions of the first arXiv version of this paper and reveals some differences between   ($\star$), ($\star\star$), ($\dagger$), ($\dagger \dagger$) and ($\bullet$). 
 
 \begin{example}\label{bitsin}($\mathsf{CH}$)
 Let $\mathbb{R}$ be the real numbers with its usual topology. Let $\mathcal{B}$ consist of open intervals of length  $\frac{1}{n}$ for   $n\in \omega-\{0\}$ i.e., $\mathcal{B}=\{(a,a+\frac{1}{n}):a\in \mathbb{R}, n\in \omega-\{0\} \}$. Note that $\mathcal{B}$ is a Noetherian $\pi$-base of $\mathbb{R}$ and $\mathcal{B}=pur(\mathcal{B})$. Then, $\mathcal{B}$ has the ($\star$), ($\star\star$), ($\dagger$) and ($\dagger \dagger$) properties but $\mathcal{B}$ does not have the  $(\bullet)$ property.

 \end{example}  
 
 \begin{proof}
To see  $\mathcal{B}$ does not have the  $(\bullet)$ property, take any $W\in \mathcal{B}$	 which has length less than 1. Note that $|\{B\in \mathcal{B}:W\subseteq B\}|=\omega_1$ and $|\mathcal{A}|\leq \omega_0$ for any disjoint family $\mathcal{A}\subseteq \tau^*(W)$.

Now, we will see that $\mathcal{B}$ has the other properties. Note that  $rank(\mathcal{B})=\omega$. Let $\bigcup_{i<\omega}\mathcal{B}_i$ be the Noetherian union of $\mathcal{B}$. Note that $\mathcal{B}_i$ consists of open intervals of length $\frac{1}{i+1}$ and $|\mathcal{B}_i|=\omega_1$ for all $i\in \omega$. Thus, we can define a Noetherian table
 $[\mathcal{B}]$ such that the  $i$-th row is a tuple $(B^0_i,B^1_i,B^2_i,\ldots)$ consisting of all elements of   $\mathcal{B}_i$ and $j<\omega_1$ for any $B^j_i$ of $(B^0_i,B^1_i,B^2_i,\ldots)$. This is possible because $|\mathcal{B}_i|=\omega_1$. Now take any $B^j_i\in [\mathcal{B}]$. Thus, $r(B^j_i)=i<\omega_0$, $o(B^j_i)=j<\omega_1$ and there exists a disjoint family $\mathcal{A}\in \tau^*(B^j_i)$ such that $|\mathcal{A}|=\omega_0$. Therefore $|\mathcal{A}|\geq \max\{r(B^j_i), |o(B^j_i)|\}$. Thus, $\mathcal{B}$ has the property $(\dagger\dagger)$. Because $\mathcal{B}=pur(\mathcal{B})$, we can define $[\mathcal{B}]=[pur(\mathcal{B})]$. Thus, we can apply the same process to $[pur(\mathcal{B})]$. Therefore,
  $\mathcal{B}$ has the property $(\dagger)$.  From Remark \ref{R:1}, $\mathcal{B}$ has the properties ($\star$) and  ($\star\star$).

 \end{proof}

  \begin{corollary}\label{C: 3. 1  sonucu}
 	
Let $X$ be a  space $\kappa$ an infinite cardinal and  $\pi w(X)\leq\kappa^+$. If $pur(\tau^*(X))=\emptyset$ or for any $U\in pur(\tau^*(X))$  there exists a disjoint family $\mathcal{A}_U\subseteq \tau^*(U)$ such that $|\mathcal{A}_U|\geq \kappa$, then we have the following:

(1) $X$ has  a $\pi$ base $\mathcal{B}$ such that  $\mathcal{B}$ satisfies the  five conditions:   ($\star$), ($\star\star$), ($\dagger$), ($\dagger \dagger$) and ($\bullet$).

(2) If \2 has a winning strategy in $BM(X)$, then \2 has a  2-tactic in $BM(X)$. 
\end{corollary}

\begin{proof} Let $\pi w(X)=\chi$ and $\mathcal{T}$ be a $\pi$-base for $X$ where $|\mathcal{T}|=\chi\le \kappa$.  From Theorem \ref{Theorem: pi Base varsa Noetherian pi Base var},  there exist a Noetherian $\pi$-base $\mathcal{B}\subseteq \mathcal{T}$ such that $\mathcal{B}$ satisfies $(i)$ and $(ii)$ of Theorem \ref{Theorem: pi Base varsa Noetherian pi Base var}.

If $pur(\tau^*(X))=\emptyset$, then from Corollary \ref{C:3},   $pur(\mathcal{B})=\emptyset$. Thus,  we get (1), and so, (2). If $pur(\tau^*(X))\neq\emptyset$, then from Corollary \ref{C:3}, $pur(\mathcal{B})\neq\emptyset$. Thus, we only consider the case $pur(\mathcal{B})\neq\emptyset$.

Firstly,  we will see that $\mathcal{B}$  satisfies $(\bullet)$-property. Now, take  any $W\in pur(\mathcal{B})$.  From the hypothesis, there exists a disjoint family $\mathcal{A}_W\subseteq \tau^*(W)$ such that $|\mathcal{A}_W|\geq \kappa$. From $(ii)$ of Theorem \ref{Theorem: pi Base varsa Noetherian pi Base var},
$|\{B\in \mathcal{B}: W\subseteq B \}|<\chi\leq\kappa^+$ for all $W\in \mathcal{B}$. Therefore, $|\{B\in \mathcal{B}: W\subseteq B \}|\leq |\mathcal{A}_W|$. Thus, $\mathcal{B}$  satisfies $(\bullet)$-property.

Now, we will see that $\mathcal{B}$  satisfies the conditions ($\star$), ($\star\star$), ($\dagger$)  and ($\dagger \dagger$).	 Because $\chi\leq\kappa^+$, from $(i)$ of Theorem \ref{Theorem: pi Base varsa Noetherian pi Base var}, we have $|\mathcal{B}|\leq \kappa^+$ and $\mathcal{B}$ has an $\alpha\times \beta$  Noetherian table $[\mathcal{B}]$ where  $\alpha, \beta\leq \kappa^+$. (Note that from Theorem \ref{yeter}, $pur(\mathcal{B})$ has also a  $\theta\times \delta$ Noetherian table $[pur(\mathcal{B})]$  satisfying  $\theta,\delta\leq \kappa^+$, because  $\emptyset\neq pur(\mathcal{B})\subseteq \mathcal{B}$). 
 
 Now, take and fix any $U\in pur(\mathcal{B})$. 
 
 (To see $\mathcal{B}$ has ($\star\star$) or ($\dagger \dagger$) conditions, say $U=B^{o(U)}_{r(U)}\in [\mathcal{B}]$. To see $\mathcal{B}$ has ($\star$) or ($\dagger$)-conditions, say $U=B^{o'(U)}_{r'(U)}\in [pur(\mathcal{B})]$ and replace  the following each $\alpha$ and $\beta$ with $\theta$ and $\lambda$, respectively. Without loss of generality we will use $r(U),o(U),\alpha$ and $\beta$. So, the proof will be valid for these  four conditions.) 
 
 From the hypothesis, we have a disjoint family $\mathcal{A}_U\subseteq \tau^*(U)$ such that $|\mathcal{A}_U|\geq \kappa$. To get ($\star$), ($\star\star$), ($\dagger$), ($\dagger \dagger$) and (2), from Corollary \ref{Teorem winning ten 2-tactige ana sonuc} and Remark \ref{R:1},  it is enough to see that $|\mathcal{A}_U|\geq  \max\{|r(U)|,|o(U)|\}$.
Because $\max\{\alpha,\beta\}\leq \kappa^+, r(U)<\alpha$ and $o(U)<\beta$,     $\max\{|r(U)|,|o(U)|\}< \max\{\alpha,\beta\}\leq\kappa^+$.  Then,  $\max\{|r(U)|,|o(U)|\} \leq \kappa$ (because $|r(U)|$ and $|o(U)|$ are cardinals). Therefore, $\max\{|r(U)|,|o(U)|\} \leq \kappa\leq |\mathcal{A}_U|$.

\end{proof}

Note that in the theorem above, same $\pi$-base  has the five conditions. The same holds for the two results that follow. At the end of this article, we will pose some questions regarding the interrelations among the five conditions.

\begin{corollary}\label{C: 3. 2  sonucu}
	
	Let $X$ be a space $\kappa$ an infinite cardinal and  $\pi w(X)\leq\kappa$. If $pur(\tau^*(X))=\emptyset$ or  $c(U)\geq \kappa$ for any $U\in pur(\tau^*(X))$, then we have the following:

	(1) $X$ has a $\pi$ base $\mathcal{B}$ such that  $\mathcal{B}$ satisfies the  five conditions:   ($\star$), ($\star\star$), ($\dagger$), ($\dagger \dagger$) and ($\bullet$).

	(2) If \2 has a winning strategy in $BM(X)$, then \2 has a  2-tactic in $BM(X)$. 
\end{corollary}

\begin{proof} (Similar arguments in the proof of Corollary \ref{C: 3. 1  sonucu} are valid for this proof. For  shortness, we did not repeat them.)

Let $\pi w(X)=\chi$ and $\mathcal{T}$ be a $\pi$-base for $X$ where $|\mathcal{T}|=\chi$.  From Theorem \ref{Theorem: pi Base varsa Noetherian pi Base var},  there exist a Noetherian $\pi$-base $\mathcal{B}\subseteq \mathcal{T}$ such that $\mathcal{B}$ satisfies $(i)$ and $(ii)$ of Theorem \ref{Theorem: pi Base varsa Noetherian pi Base var}.
 
 We only consider the case $pur(\mathcal{B})\neq\emptyset$.

Firstly,  we will see that $\mathcal{B}$  satisfies $(\bullet)$-property. Now, take  any $W\in pur(\mathcal{B})$.  From the hypothesis, $c(W)\geq \kappa.$ From $(ii)$ of Theorem \ref{Theorem: pi Base varsa Noetherian pi Base var},
$|\{B\in \mathcal{B}: W\subseteq B \}|<\chi\leq\kappa$. Therefore, $|\{B\in \mathcal{B}: W\subseteq B \}|< \chi\le c(W)$.  Thus, $|\{B\in \mathcal{B}: W\subseteq B \}|<  c(W)$. Because of the definition of $c(W)$, there exists a disjoint family $\mathcal{A}_W\subseteq \tau^*(W)$ such that $|\mathcal{A}_W|\geq |\{B\in \mathcal{B}: W\subseteq B \}|$. Thus, $\mathcal{B}$  satisfies $(\bullet)$-property.

Now, we will see that $\mathcal{B}$  satisfies the conditions ($\star$), ($\star\star$), ($\dagger$)  and ($\dagger \dagger$).	 Because $\chi\leq\kappa$, from $(i)$ of Theorem \ref{Theorem: pi Base varsa Noetherian pi Base var}, we have $|\mathcal{B}|\leq \kappa$ and $\mathcal{B}$ has an $\alpha\times \beta$  Noetherian table $[\mathcal{B}]$ where  $\alpha, \beta\leq \kappa$.
Take any $U\in pur(\mathcal{B})$. Because $r(U)<\alpha\leq \kappa$ and $o(U)<\beta\leq \kappa$, $r(U),o(U)<\kappa$. Since $c(U)\geq \kappa$, there exists a disjoint family $\mathcal{A}_U\subseteq \tau^*(U)$ such that $|\mathcal{A}_U|\geq \max\{r(U),o(U)\}$. Thus, because of  the similar arguments in the proof of Corollary \ref{C: 3. 1  sonucu}, $\mathcal{B}$ has the  four conditions:  ($\star\star$), ($\dagger \dagger$), ($\star$) and ($\dagger$). So, we have (2) and the proof is completed.
\end{proof}

\begin{theorem}\label{C:2}

	For any space $X$ if $\pi w(X)\leq \omega_1$, then we have the following.

	(1)  $X$ has a $\pi$ base   satisfying the  five conditions:  ($\star$), ($\star\star$), ($\dagger$), ($\dagger \dagger$) and ($\bullet$).

	(2) If \2 has a winning strategy in $BM(X)$, then \2 has a  2-tactic in $BM(X)$.
\end{theorem}

\begin{proof} This follows immediately from Corollary \ref{C: 3. 1  sonucu} and definition of purification ($pur(\mathcal{A})$).

\end{proof}

Note that there is no separation axiom  on the spaces in the Theorem  above.

\begin{lemma}\label{Lemma: GCH ile w(X)leq c(X)+}($\mathsf{GCH}$) 
	Let $X$ be a $T_3$ space.  If $d(X)=c(X)$, then $w(X)\leq c(X)^+$ (so, $\pi w(X)\leq c(X)^+$).
\end{lemma}
\begin{proof}
	From Lemma 2.7.(b) in \cite{juhasz1980cardinal} , $w(X)\leq 2^{d(X)}$. Then  $w(X)\leq 2^{d(X)}=2^{c(X)}=c(X)^+$ with GCH.
\end{proof}

\begin{theorem}\label{Theorem: Seperable uzay Noetherian pi base omega1 den  semali var}($\mathsf{CH}$) Let $X$ be a separable $T_3$ space.

	(1)  $X$ has a $\pi$ base   satisfying the conditions of Corollary \ref{Teorem winning ten 2-tactige ana sonuc} and  conditions of Theorem \ref{Teorem winning ten 2-tactige ana sonuc 2}.

	(2) If \2 has a winning strategy in $BM(X)$, then \2 has a  2-tactic in $BM(X)$.
	
\end{theorem}
\begin{proof}
 Because $X$ is separable space,  $d(X)=c(X)=\omega$. From CH and Lemma \ref{Lemma: GCH ile w(X)leq c(X)+}, $w(X)\leq c(X)^+= \omega_1$. So, $\pi w(X)\leq  \omega_1$. Then, from Corollary \ref{C:2}, we get (1) and  (2).
\end{proof}

The following Theorem \ref{Teorem Erdos-Tarski} and Remark \ref{C:1} are known.

\begin{theorem}\label{Teorem Erdos-Tarski} (12.2. Theorem in \cite{kunen2014handbook} or 4.1. in \cite{juhasz1980cardinal}) 
	Let $X$ be a space and $c(X)=\kappa$. If $\kappa$ is a singular cardinal, 
	then $X$ has a cellular family of cardinality $\kappa$.\begin{flushright}
		\kare
	\end{flushright} 
\end{theorem}

It is known that any cardinal must be regular or singular, any regular cardinal must be limit or successor and $X$ has a cellular family of cardinality $c(X)$ if $c(X)$  is a successor cardinal. So, from 
Theorem \ref{Teorem Erdos-Tarski}, we
have the following corollary.

\begin{remark}\label{C:1}
	For any space $X$ if $c(X)=\kappa$ is not a regular limit cardinal, then $X$ has a cellular family of cardinality $\kappa$.
  
	\begin{flushright}
		\kare
	\end{flushright}  	
\end{remark}

By using Corollary \ref{C: 3. 1  sonucu} and supposing that $c(U)$ is not a regular limit cardinal where $U\in pur(\tau^*(X))$, we have the following corollary. If $c(U)$ is a regular limit cardinal, we can employ  Corollary \ref{C: 3. 2  sonucu}.

\begin{corollary}\label{Corollary altuzayda cellurity ayni ise 2 tactik var}

	Let $X$ be a space, $pur(\tau^*(X))\neq \emptyset$ and  $\pi w(X) \leq c(X)^+$. If $c(X)$ is not a regular limit cardinal and   $c(U)=c(X)$ for any $U\in pur(\tau^*(X))$, 
	then \2 has a  2-tactic in $BM(X)$ whenever \2 has a winning strategy in $BM(X)$. 
	
\end{corollary}

\begin{proof}
	Let $c(U)=c(X)=\kappa$. From Remark \ref{C:1}, there exists a disjoint family $\mathcal{A}_U\subseteq \tau^{*}(U)$ such that $|\mathcal{A}_U|= \kappa$. Thus, from Corollary \ref{C: 3. 1  sonucu}, \2 has a  2-tactic in $BM(X)$.
\end{proof}

 \begin{corollary}\label{Corollary altuzayda cellurity density ayni   ise 2 tactik var}($\mathsf{GCH}$)
 	Let $X$ be a $T_3$  space and $\kappa$   not a regular limit cardinal and $pur(\tau^*(X))\neq \emptyset$. Suppose   for any $U\in pur(\tau^*(X))$,   $d(X)=c(X)=c(U)=\kappa$. Then, \2 has a   2-tactic in $BM(X)$ whenever \2 has a winning strategy in $BM(X)$. 
 	
 \end{corollary}

 \begin{proof}
 	From Lemma \ref{Lemma: GCH ile w(X)leq c(X)+}, $\pi w(X)\leq c(X)^+$. Thus, from Corollary \ref{Corollary altuzayda cellurity ayni ise 2 tactik var}, \2 has a  2-tactic in $BM(X)$.

 \end{proof}

  The following lemma is known and easy to prove.  
  
  \begin{lemma} \label{Lemma kazanma varsa maksimal cellular family}
  	Let $X$ be a space and $\mathcal{A}\subseteq \tau^*(X)$ a maximal cellular family.
  	Then \2 has a winning strategy (k-tactic) in $BM(X)$ iff \2 has a winning strategy (k-tactic) in $BM(A)$ for every $A\in \mathcal{A}$.
  	
  \end{lemma}

 \begin{theorem}\label{ke2y}
 	Let $X$ be a space. There exists  a maximal cellular family $\mathcal{A}\subseteq \tau^*(X)$  such that for any $A\in \mathcal{A}$ and for any  $B\in \tau^{*}( A)$,  $c(B)=c(A)$. In addition, assuming  GCH and letting $X$ be a $T_3$ space, for this family $\mathcal{A}$ and for any $A\in \mathcal{A}$ if $c(A)$ is not a regular limit cardinal  and $d(A)=c(A)$, then \2 has a  2-tactic in $BM(X)$ whenever  \2 has a winning strategy in $BM(X)$. 
 \end{theorem}
 
 \begin{proof}

  For any $U\in \tau^*(X)$, let $min\{c(B): B\in \tau^*(U)\}=\alpha_U$ and choose  a $W_U\in \tau^*(U)$ such that $c(W_U)=\alpha_U$. 
 Now, define $\mathcal{A}^*=\{ W_U:U\in \tau^*(X)\}$. From Lemma \ref{key1}, there exists a maximal cellular subfamily $\mathcal{A}$ of  $\mathcal{A}^*$. 
 Therefore, for any $A\in \mathcal{A}$ and for any nonempty open subset $B\subseteq A$,  $ c(B)=c(A)$.

 	For the second part, fix any $A\in \mathcal{A}$. Suppose \2 has   a winning strategy in $BM(X)$. From Lemma \ref{Lemma kazanma varsa maksimal cellular family}, \2 has  a winning strategy in $BM(A)$. 
 	 So, from Corollary \ref{Corollary altuzayda cellurity density ayni   ise 2 tactik var} (from Corollary \ref{c:makale} if $pur(\tau^*(A))=\emptyset$),  \2 has  a  2-tactic  in $BM(A)$. Since $A$ is arbitrary, from Lemma \ref{Lemma kazanma varsa maksimal cellular family}, \2 has a   2-tactic in $BM(X)$. 
 \end{proof}

Monotonically normal spaces are well known. 
A $T_1$ space $X$ is said to be monotonically normal if there is a function $F(., .)$
which assigns to each point $x\in X$ and each open neighbourhood $U$ of $x$ an open subset $F(x, U)$
containing $x$ which satisfies the following conditions.

(1) $F(x, U) \subseteq  F(x, V)$ whenever  $U \subseteq  V$,

(2) $F(x, X\setminus\{y\}) \cap  F(y, X\setminus\{x\})) =\emptyset$ for every distinct points  $x, y\in X$.

Arbitrary subspaces of monotonically normal spaces are again monotonically
normal \cite{gartside1997cardinal}. 
Every linearly ordered space is monotonically normal (\cite{gartside1997cardinal}, 5.3 in \cite{heath1973monotonically}).

\begin{theorem}\label{tartık} Assuming there is no regular limit cardinal, if $X$ has a monotonically normal compactification and \2 has a winning strategy in $BM(X)$, then \2 has a  2-tactic in $BM(X)$. 
\end{theorem}
 \begin{proof} From the first part of Theorem \ref{ke2y},  there exists  a maximal cellular family $\mathcal{A}\subseteq \tau^*(X)$  such that for any $A\in \mathcal{A}$ and for any  $B\in \tau^{*}( A)$,  $c(B)=c(A)$.
 	
Because $X$ has a  monotonically normal compactification, $X$ is subspace of monotonically normal space. Then, $X$ is  monotonically normal space. Thus, for any $A\in  \mathcal{A}$, $A$ is monotonically normal and   has a  monotonically normal compactification (Note that the closure of ${A}$ in the compactification of $X$ is compact).  Thus, from	Corollary 19 in \cite{gartside1997cardinal},  $\pi w(A) =  d(A)$, and from Theorem A in \cite{gartside1997cardinal}, $d(A)\le c(A)^+$. Therefore,  $\pi w(A) \le c(A)^+$. 

 From Lemma \ref{Lemma kazanma varsa maksimal cellular family}, \2 has  a winning strategy in $BM(A)$. Because $\pi w(A) \le c(A)^+,$ from Corollary \ref{Corollary altuzayda cellurity ayni ise 2 tactik var} (if $pur(\tau^*(A))=\emptyset$, from Corollary \ref{c:makale}), \2 has  a 2-tactic in $BM(A)$. Because $A$ is an arbitrary element of $\mathcal{A}$,  from Lemma \ref{Lemma kazanma varsa maksimal cellular family}, \2 has  a 2-tactic in $BM(X)$.

 \end{proof}

Note that the assumption of the nonexistence of a regular limit cardinal in the theorem above and the following corollary is not entirely essential. It is enough that $c(A)$ is not a regular limit cardinal for all $A\in \mathcal{A}$.

\begin{corollary}
Assume that no regular limit cardinal exists, and let 
$X$ be either a linearly ordered space or, more generally, a generalized ordered space (GO-space). If \2 has a winning strategy in $BM(X)$, then \2 has a  2-tactic in $BM(X)$. 
\end{corollary}

\begin{proof}

If $X$ is a GO-space, then $X$ is a subspace of a linearly ordered  space  $Y$. 
Then, $Y$ has a linearly ordered compactification $Z$ because  every linearly ordered  space has a linearly ordered   campactfication  (see 3.12.3 (b) in \cite{engelking1989general} or Example 8.3 in \cite{nagata1985modern}). Because every linearly ordered space is monotonically normal,    $Z$ is  monotonically normal.
 Consequently, the closure of $X$ in $Z$ is both compact and monotonically normal. Therefore, the closure of $X$ in $Z$ is a monotonically normal compactification of $X$. Thus, from Theorem \ref{tartık},   \2 has a   2-tactic in $BM(X)$.

\end{proof}

\section{Open Questions}\label{Ack sorular}

The following question is related to Theorem \ref{Prop: }.

\begin{question}

Let $\mathcal{B}$ be an arbitrary Noetherian  base of a space $X$ and $rank(\mathcal{B})$  not a cardinal. Then, is there any  Noetherian base $\mathcal{B}'\subseteq \mathcal{B}$ such that $rank(\mathcal{B}')\leq |rank(\mathcal{B})|$?

\end{question}

Where $\kappa$ is a cardinal number, in the following two questions: $\kappa^{+1}=\kappa^+$ and for $n\geq 2$,  $\kappa^{+n}=(\kappa^{+(n-1)})^+$.

The following question is related to the Galvin's theorem.  

\begin{question}\label{soru}
	If \2 has a winning strategy in  $BM(X)$ where $X$ has a $\pi$-base $\mathcal{B}$ with the following ($\Join$)-property, does \2 have an  $(n+2)$-tactic in $BM(X)$? 
	
	($\Join$) For any $U\in pur(\mathcal{B})$  there exists a disjoint family 
	$\mathcal{A}_U\subseteq \tau^*(U)$  such that $|\mathcal{A}_U|^{+n}\geq  |\{B \in\mathcal{B}: U \subseteq B \}|$. 	
\end{question}

The following question is related to Theorem \ref{Teorem winning ten 2-tactige ana sonuc 2}

\begin{question}
If \2 has a winning strategy in $BM(X)$ where $X$ has a $\pi$-base $\mathcal{B}$ with the following ($\star_n$)-property, does \2 have an  $(n+2)$-tactic in $BM(X)$?

($\star_n$) $pur(\mathcal{B})=\emptyset$  or  $pur(\mathcal{B})$ is   Noetherian and has a 
Noetherian table $[pur(\mathcal{B})]$ such that
for any $B\in [pur(\mathcal{B})]$
there exists a disjoint family $\mathcal{A}_{B}\subseteq \tau^*(B)$ such that 
$|\mathcal{A}_{B}|^{+n}\geq \max\{|r(B)|,M_B\}$ where Noetherian union $pur(\mathcal{B})$ is $\bigcup_{\theta<\alpha}\mathcal{B}_\theta$ and $M_B=\sup\{|\{V\in  \mathcal{B}_\theta:B\subseteq V, o(V)\leq o(B)\}|:\theta\leq r(B)\}$. 
	
\end{question}

In this paper, we prove that Corollary \ref {Teorem winning ten 2-tactige ana sonuc} implies Theorem \ref{Teorem winning ten 2-tactige ana sonuc 2}, and    Galvin's Theorem implies  Theorem \ref{Teorem winning ten 2-tactige ana sonuc 2}. Now, we pose some questions about these relationships. We conjecture that the answers to the following first three questions are affirmative; in other words, we think that  Theorem \ref{Teorem winning ten 2-tactige ana sonuc 2} does not imply Galvin's Theorem, and Corollary \ref {Teorem winning ten 2-tactige ana sonuc}  and Galvin's Theorem are distinct.

\begin{question}\label{Galvin}
	
Is there any space $X$ having a $\pi$-base which has  ($\star$) or ($\star\star$)-conditions of Theorem \ref{Teorem winning ten 2-tactige ana sonuc 2}, but $X$  has no  $\pi$-base  $\mathcal{B}$ which has the following ($\bullet$)-property?

($\bullet$) $pur(\mathcal{B})=\emptyset$ or for any $U\in pur(\mathcal{B})$ there exists a disjoint family 
$\mathcal{A}_U\subseteq \tau^*(U)$  such that $|\mathcal{A}_U|\geq  |\{B \in\mathcal{B}: U \subseteq B \}|$. 

\end{question}

The following questions are related to Corollary \ref{Teorem winning ten 2-tactige ana sonuc} and the Galvin's theorem.

\begin{question}
	Is there any space $X$ having a $\pi$-base which has ($\dagger$) or ($\dagger\dagger$)-conditions of Corollary \ref{Teorem winning ten 2-tactige ana sonuc}, but $X$  has no  $\pi$-base  $\mathcal{B}$ which has  ($\bullet$)-property in  Question \ref{Galvin}?
	
\end{question}

\begin{question}
	Is there any space $X$ having a $\pi$-base  which has ($\bullet$)-property in  Question \ref{Galvin}, but $X$  has no  $\pi$-base  which has   ($\dagger$) or ($\dagger\dagger$)-conditions of Corollary \ref{Teorem winning ten 2-tactige ana sonuc}?
	
\end{question}

The following can be thought as a project, rather than a specific question.
\begin{question}\label{de}
	Let $X$ be space and \2 has a winning strategy in $BM(X)$. Does there exist a "property" that characterizes the following equivalence:
	
	\2 has a 2-tactic in $BM(X)$ if and only if there exists a $\pi$-base for $X$ satisfying this "property"?.
	
Is there a corresponding set-theoretic statement for this "property", and what kind of properties does this statement have?
\end{question}

\bibliographystyle{plain}

\end{document}